\documentclass[11pt,reqno]{amsart}
\usepackage{amscd,amsmath,amsopn,amssymb,amsthm,multicol,xcolor,mathtools,tikz-cd}

\newcommand\com[1]{}
\newcommand\C{{\mathbb C}}

\newcommand\g{{\mathfrak g}}

\newcommand\Hmod{{\tilde{\mathrm{H}}}}
\newcommand\Hcech{{\check{\mathrm{H}}}}
\newcommand\HH{{\mathrm{H}}}
\newcommand\op[1]{\mathop{\rm #1}\nolimits}
\newcommand\p{\partial}
\newcommand\R{{\mathbb R}}

\newcommand\vf{\mathcal{D}}
\newcommand\J{\mathbf{J}}

\theoremstyle{plain}
\newtheorem*{theorem*}{Theorem}
\newtheorem{prop}{Proposition}
\newtheorem{theorem}[prop]{Theorem}
\newtheorem{cor}{Corollary}
\newtheorem{lemma}{Lemma}
\theoremstyle{definition}
\newtheorem{definition}{Definition}
\newtheorem{example}{Example}
\newtheorem{remark}{Remark}

\begin{document}

\title[]{Invariant divisors and equivariant line bundles}

\author[Boris Kruglikov]{Boris Kruglikov$^\dagger$}
\address{$^\dagger{}$ Department of Mathematics and Statistics, UiT the Arctic University of Norway, Troms\o\ 9037, Norway.}

\author[Eivind Schneider]{Eivind Schneider$^\dagger$}

\address{Email addresses:\qquad {\tt boris.kruglikov@uit.no}\quad\text{\rm and }\quad {\tt eivind.schneider@uit.no}\hspace{1pt}.}

% \email{ boris.kruglikov@uit.no\quad\text{\rm and }\quad eivind.schneider@uit.no\hspace{1pt}.}

 \begin{abstract}
Scalar relative invariants play an important role in the theory of group actions on a manifold as their zero sets 
are invariant hypersurfaces. 
Relative invariants are central in many applications, where they often are treated locally since
an invariant hypersurface may not be a locus of a single function.
Our aim is to establish a global theory of relative invariants.

For a Lie algebra $\g$ of holomorphic vector fields on a complex manifold $M$, any holomorphic 
$\g$-invariant hypersurface is given in terms of a $\g$-invariant divisor. 
This generalizes the classical notion of scalar relative $\g$-invariant.
Any $\g$-invariant divisor gives rise to a $\g$-equivariant line bundle, and a large part of this paper is therefore 
devoted to the investigation of the group $\mathrm{Pic}_\g(M)$ of $\g$-equivariant line bundles. We give a 
cohomological description of $\mathrm{Pic}_\g(M)$ in terms of a  double complex interpolating 
the Chevalley-Eilenberg complex for $\g$ with the \v{C}ech complex of the sheaf of holomorphic functions on $M$.

We also obtain results about polynomial divisors on affine bundles and jet bundles. This has applications 
to the theory of differential invariants.
Those were actively studied in relation to invariant differential equations, but the description of multipliers 
(or weights) of relative differential invariants was an open problem. We derive a characterization of them 
with our general theory. Examples, including projective geometry of curves and second-order ODEs, not only 
illustrate the developed machinery, but also give another approach and rigorously justify some 
classical computations. At the end, we briefly discuss generalizations of this theory.
 \end{abstract}

\maketitle
\tableofcontents
\bigskip

%%%%%%%%%%
\section{Introduction}

\subsection{Background on relative invariants}

Consider a manifold $M$ together with a Lie group $G$ acting on $M$.
Let $\mathcal{F}(M)$ the algebra of functions on $M$ and $\mathcal{F}(M)^\times$ the
multiplicative subgroup of nonvanishing functions. The action of $g \in G$ on $M$ induces the pullback (right) action $g^*$ on $\mathcal{F}(M)$.
A (scalar) relative invariant is a function $R \in \mathcal{F}(M)$ satisfying
 \[ 
g^* R = \Lambda(g) R \qquad \forall g \in G,
 \]
for some map $\Lambda \colon G \to \mathcal{F}(M)^\times$, called the multiplier, or weight, of $R$. 
If $\g \subset \vf(M)$ denotes the Lie algebra of vector fields on $M$ corresponding to the Lie group action, then $R$ also satisfies
 \[ 
X(R) = \lambda(X) R \qquad \forall X \in \g,
 \]
for some (infinitesimal) multiplier $\lambda \in \g^* \otimes \mathcal{F}(M)$,
or weight, of $R$. It follows from the definition that
the locus $\{R=0\} \subset M$ is $G$-invariant (resp.\ $\g$-invariant).

In the case $\Lambda=1$ (resp.\ $\lambda=0$), the function $R$ is called an absolute invariant,
and each level set $\{R=\mathrm{const}\} \subset M$ is invariant, so that we get an invariant foliation of $M$.
Absolute invariants are well understood in several different settings, see
\cite{W,M,O1,SR} for the classical invariant theory and \cite{O,KL} for its differential counter-part.

For example, in the case of a regular smooth Lie group action on a smooth manifold, locally by the Frobenius theorem,
the number of functionally independent absolute invariants is equal to the codimension of an orbit,
and orbits are locally separated by that many invariants (see, for example, Chapter 2 of \cite{O}).
In the case of an algebraic group action on an algebraic variety, globally by the Rosenlicht theorem,
orbits in general position are separated by rational absolute invariants, and the number of
algebraically independent rational absolute invariants is equal to the codimension of a generic orbit
(see, for example, Chapter 13 of \cite{SR}).

Relative invariants with nontrivial weight are less understood, although they appear in many important applications
(we refer to the introduction to \cite{FO} and also to the more recent \cite{O2}).
In particular, they are often used to describe $\g$-invariant hypersurfaces containing singular orbits.
An infinitesimal multiplier $\lambda$ is a 1-cocycle of the Chevalley-Eilenberg complex of $\g$ with coefficients
in $\mathcal{F}(M)$. Relationships between the weights of relative (differential) invariants
and the Chevalley-Eilenberg cohomology was discussed in \cite{COW,O}.
The question of realizability of a given cocycle as the weight of some relative invariant was answered locally
in the case of a regular smooth $G$-action and $\mathcal{F}(M)=C^\infty(M)$
by M.\,Fels and P.\,Olver (\cite{FO} and \cite[Th.\,3.36]{O}),
also in the context of vector-valued relative invariants. In the general case the answer is not known.

Note that rescaling of $R$ by a non-zero function $e^f$, $f\in\mathcal{F}(M)$, changes $\lambda$ by a
coboundary $df$, which naturally associates the Chevalley-Eilenberg cohomology class $[\lambda]\in\HH^1(\g,\mathcal{F}(M))$
to the  (equivalence class of the) relative invariant $R$. A proper version of this cohomology will be central in our work.

\subsection{A setup for global invariants}

In general, the description of invariant hypersurfaces (analytic subvarieties of codimension 1) by relative invariants works only locally: 
there exist invariant hypersurfaces that cannot be described globally as 
the zero locus of a relative invariant. In this paper we restrict to holomorphic actions on complex manifolds, 
where this problem can be solved using the language of divisors. Some results extend to real analytic 
and algebraic situations, but smooth versions of our global results in general are not available. 
Thus we specialize our algebra of functions $\mathcal{F}(M)$ to consist of holomorphic functions,
and we will work with the sheaf $\mathcal{O}=\mathcal{O}_M$ of such functions on a complex manifold $M$.

In most of the paper we will concentrate on the infinitesimal (Lie algebra) picture as it is conceptually simpler 
and lends itself well to computations. Moreover, it is more general, as a Lie group action always gives rise 
to a Lie algebra of (complete) vector fields, but not every Lie algebra action
can be integrated (the manifold $M$ is not assumed compact; the Lie algebra may be infinite-dimensional).
It should be noted that for algebraic groups $G$ (as well as for compact Lie groups)
 % these latter have found applications in equivariant algebraic topology and stable homotopy category
the equivariant line bundles have been well studied, see \cite[Ch.\,1.3]{M} and \cite[\S4.2]{B} for the definition
and properties of the $G$-equivariant Picard group $\op{Pic}_G(M)$ in the context of algebraic geometry.
 % (applications to homotopy categories are not relevant for our work)
Our setup is more general, and we present the corresponding theory for Lie groups in Section \ref{sect:Gequivar}.
The main object of study, however, will be the Picard group $\op{Pic}_\g(M)$ of $\g$-equivariant 
line bundles defined for any Lie algebra $\g$ of holomorphic vector fields on $M$.

A divisor $D$ on $M$ is given by a collection of meromorphic functions $f_\alpha$ defined on each chart in an open cover
$\{U_\alpha\}$ of $M$ (if the functions $f_\alpha$ are holomorphic, then $D$ is called effective). The functions $f_\alpha$ are required to be consistent, in the sense that the zeros and poles of $f_\alpha$ and $f_\beta$ agree 
on $U_\alpha \cap U_\beta$, which is equivalent to $f_\alpha/f_\beta$ being a nonvanishing holomorphic function 
on $U_\alpha \cap U_\beta$. 
(Our $D$ correspond to Cartier divisors, which are equivalent to Weyl divisors 
for the nonsingular analytic varieties we consider.)
Analytic hypersurfaces of a complex manifold $M$ are given locally by the vanishing 
of a holomorphic function and globally by an effective divisor.
% (Our $D$ correspond to Cartier divisors, but for nonsingular analytic varieties we consider the distinction 
 % with Weyl divisors
% is not essential.)

If $\g$ is a Lie algebra of vector fields on $M$ and $N \subset M$ is a $\g$-invariant hypersurface defined
by the divisor $D= \{f_\alpha\}$, then each vector field of $\g$ is tangent to $N$, implying that  for each $\alpha$
 \[ 
X(f_\alpha) = \lambda_\alpha(X) f_\alpha \qquad \forall X \in \g,
 \]
for some weight $\lambda_\alpha \in \g^* \otimes \mathcal{O}(U_\alpha)$, which is a 1-cocycle in the
Chevalley-Eilenberg complex of $\g$ with coefficients in the $\g$-module $\mathcal{O}(U_\alpha)$ of holomorphic functions on $U_\alpha \subset M$. Such a divisor is called $\g$-invariant.
%Change of coordinates in the line bundle results in change of the weight $\lambda_\alpha$
Multiplying each $f_\alpha$ by nonvanishing holomorphic functions gives a different representative of the same divisor, and the weight $\lambda_\alpha$ is in this case changed by a coboundary,
so the weights can be identified with elements in the Chevalley-Eilenberg cohomology
$\HH^1(\g,\mathcal{O}(U_\alpha))$  or, more precisely, a slightly modified version thereof. A collection of such weights, or multipliers, for each element
of the cover $\{U_\alpha\}$, that are compatible on overlaps, yields a multiplier group
that we will denote $\mathfrak{M}_\g(M)$. Below we will define it in terms of a certain double complex.
 % If the open cover $\{U_\alpha\}$ is sufficiently nice ...

As is well known, any divisor $D$ on $M$ gives rise to a line bundle $[D] \to M$, with transition functions 
$g_{\alpha \beta} = f_\alpha/ f_\beta$, on which $f_\alpha$ are local defining functions of a particular section
(and, geometrically, $D$ is the locus of this section).
When $D$ is $\g$-invariant, then there exists a lift of the Lie algebra $\g \subset \vf(M)$ to a Lie algebra 
$\g^\lambda \subset \vf([D])$ defined locally in terms of the weights $\lambda= \{ \lambda_\alpha\}$ of $D$, 
meaning that $([D],\g^\lambda)$ is a $\g$-equivariant line bundle.
Properly localized, the obstruction for such a lift,  and thus for the existence of invariant divisors, belongs in general to the equivariant Picard group $\op{Pic}_\g(M)$.

\subsection{Overview of the novel results}

Due to a close relationship between $\g$-invariant divisors and $\g$-equivariant line bundles,
Section \ref{sect:2.1} starts with an investigation of prerequisites for the latter.
The Picard group $\mathrm{Pic}(M)$, consisting of holomorphic line bundles over $M$ up to equivalence, 
is isomorphic to the \v{C}ech cohomology group $\Hcech^1(M,\mathcal{O}^\times)$.
 
In order to describe the group $\mathrm{Pic}_\g(M)$ of $\g$-equivariant line bundles, we unite the \v{C}ech complex with 
the Chevalley-Eilenberg complex into a double complex $C^{\bullet,\bullet}$. 
 % This unification, and several of its consequences, is explored in Section \ref{sect:gequivar}. 
The direct limit of the first total cohomology of this complex (also called hypercohomology, cf.\ \cite{Gr}) 
is exactly the desired group: $\mathrm{Pic}_\g(M):=\varinjlim \HH^1(\mathrm{Tot}^\bullet(C))$. 

There exist natural homomorphisms $\Phi_1\colon \op{Pic}_\g(M)\to\op{Pic}(M)$ and 
$\Phi_2\colon \op{Pic}_\g(M)\to\mathfrak{M}_\g(M)$. The image 
of $\varpi := \Phi_1\times\Phi_2$ in $\op{Pic}(M)\times\mathfrak{M}_\g(M)$ defines the reduced Picard group
 \[
\mathrm{Pic}_\g(M)\stackrel{\varpi}\longrightarrow\mathrm{Pic}^{\op{red}}_\g(M)\to 0, 
 \]
 whence a double homomorphism $(\Psi_1,\Psi_2)$
such that $\Psi_i\circ\varpi=\Phi_i$ and $\ker\Psi_1\cap\ker\Psi_2=0$:
 \[\begin{tikzcd}
& &\op{Pic}^{\op{red}}_\g(M)\arrow{ld}[swap]{\Psi_1}\arrow{rd}{\Psi_2} \\
& \op{Pic}(M) & & \mathfrak{M}_\g(M)
 \end{tikzcd}\]
 \begin{theorem} \label{Th1}
 The group $T_\g(M):=\ker(\varpi)$ of equivariant line bundles with trivial reduction is defined by \eqref{eq:TgM} and consists of the global lifts of $\g$ to the trivial line bundle over $M$ that are locally trivial, modulo globally trivial lifts. 
  \end{theorem}
  %We will describe $T_\g(M):=\ker(\varpi)$ in Proposition \ref{prop:HtotPic}. 
When $T_\g(M)=0$, $\Phi_1\times\Phi_2$ embeds $\op{Pic}_\g(M)$ in $\op{Pic}(M)\times\mathfrak{M}_\g(M)$ (Corollary \ref{cor:isomorphic} gives two sufficient conditions for this);
generally the same is true for $\op{Pic}_\g(M)/T_\g(M)$.

The homomorphisms $\Psi_1,\Psi_2$ (and likewise $\Phi_1,\Phi_2$) are neither injective nor surjective, in general.
We will describe $\ker(\Psi_i)$ and $\op{im}(\Psi_i)$ in terms of the iterated cohomology 
of the double complex $C^{\bullet,\bullet}$. In particular, we will show that 
under certain topological conditions, if the isotropy algebra $\g_p$ of a generic point $p\in M$ is perfect,
then $\ker(\Psi_1)=0$ and $\op{Pic}_\g(M)\subset\op{Pic}(M)$.
This is an infinitesimal version of Proposition 1.4 from \cite{M}, which gives 
sufficient conditions for an algebraic group $G$ to admit at most one linearization on any line bundle.
The following statements elaborate on the cases considered in \cite{M} and \cite{FO} respectively.

 \begin{cor}
{\rm (i)} If $\mathfrak{M}_\g(M)=0$ and $T_\g(M)=0$ then $\Phi_1\colon\op{Pic}_\g(M)\to\op{Pic}(M)$ is injective.\\
{\rm (ii)} Likewise, if $\op{Pic}(M)=0$ and $T_\g(M)=0$ then $\Phi_2\colon \op{Pic}_\g(M)\to\mathfrak{M}_\g(M)$ is injective.
\end{cor}

In Section \ref{sect:invariantdivisors} we consider the homomorphism
 \[
j_\g\colon\mathrm{Div}_\g(M) \to \mathrm{Pic}_\g(M)
 \]
mapping a $\g$-invariant divisor $D$ with weight $\lambda$ to the $\g$-equivariant line bundle $([D],\g^\lambda)$.
The canonical morphism $j\colon \op{Div}(M)\to\op{Pic}(M)$,  which takes $D$ to $[D]$, is well-understood: its kernel and cokernel are given by exact sequence \eqref{eq:longexactpic};
for smooth projective varieties $j$ is epimorphic and $\op{Pic}(M)$ corresponds to the 
class group $\op{Cl}(M)$ of equivalent divisors, cf.\ \cite{GH}.
 % divisor - line bundle correspondence 
In contrast, even in the smooth projective case, the map $j_\g$ is generally neither injective nor surjective.

We will give a necessary criterion for a $\g$-equivariant line bundle
$(L \to M, \hat \g)$, where $\hat \g$ is a lift of $\g$ to $L$, to be the image of a $\g$-invariant divisor, namely that 
generic $\hat \g$-orbits on $L$ project bijectively (in our setup: biholomorphic) to $\g$-orbits on $M$ 
(projection may be non-injective on singular orbits). We call such Lie algebras transversal,
borrowing the terminology from  \cite{AF}, although their notion of transversality was a slightly stronger requirement. 
%In the context of Lie group actions this property was 
%called transversality in \cite{AF} (under the assumption of regular action), and we keep the same convention 
%in the more general context of Lie algebras.
%{\color{blue} This notion is also important to interpret the space $T_\g(M)$, used above, as the group of global transveral $\g$-lifts to the trivial line bundle over $M$.}

 \begin{theorem}\label{Th2}
If $D=\{f_\alpha\}$ is a $\g$-invariant divisor and $\lambda=\{\lambda_\alpha\}$ is the corresponding
weight, then the lift $\g^\lambda \subset \vf([D])$ defined by $\lambda$ is transversal.
 \end{theorem}
 
Thus if $\hat\g\subset\vf(L)$ is not transversal, then the $\g$-equivariant line bundle 
$(L\to M,\hat \g)$ is not in $\op{im}(j_\g)$. The condition $(L,\hat\g)\in\op{im}(j_\g)$ restricts not only $\hat \g$, 
but also $L$ via $\op{im}(\Psi_1\circ j_\g)\subset\op{im}(j)$.

The proof of Theorem \ref{Th2} is based on a local argument and is similar to that of \cite[Th.\,3.36]{O}
and \cite[Th.\,5.4]{FO}, where lifts of $\g$ to the trivial bundle are considered. 
It is important to note that in our general setting, contrary to the local regular settings of \cite{FO,O}, 
this criterion is only necessary but not sufficient, which will be illustrated in examples.
Yet, in an algebraic context the converse statement holds true, up to an integer factor for the degree (see Theorem \ref{th:orbitdimR}). 

In Section \ref{sect:Gequivar} we show that the group of $G$-equivariant line bundles can be described 
by a certain Lie group cohomology with coefficients in the sheaf $\mathcal{O}^\times$, which combines the
\v{C}ech cohomology of $\mathcal{O}^\times$ and the continuous Lie group cohomology with coefficients 
in the $G$-module $\mathcal{O}^\times(M)$. This in turn is related to the equivariant Picard group 
$\op{Pic}_G(M)$, studied before in particular situations when $G$ is algebraic or compact.
We also discuss its relation to $\op{Pic}_\g(M)$.

Several examples of computation are spread throughout Section \ref{sect:divisorsandbundles}, 
demonstrating global constraints in the theory of $\g$-invariant divisors and $\g$-equivariant line bundles.
For instance, when $M=\mathbb CP^1$ with the standard coordinate charts $U_0, U_\infty \subset \mathbb CP^1$,
and $\g=\mathfrak{aff}(1,\mathbb C)$ is the 2-dimensional Lie subalgebra of
$\mathfrak{sl}(2,\mathbb C) \subset \vf(M)$, then
 \[
\mathrm{Pic}_\g(U_0) \simeq \HH^1(\g,\mathcal{O}(U_0)) = \mathbb C, \qquad \mathrm{Pic}_\g(U_\infty) \simeq  \HH^1(\g,\mathcal{O}(U_\infty)) = \mathbb C^2.   
 \]
The isomorphism between the group of $\g$-equivariant line bundles and the Chevalley-Eilenberg cohomology group
follows from the fact that all line bundles over $\mathbb C$ are trivial. On $\mathbb CP^1$, on the other hand,
there are only countably many line bundles, namely $\mathcal{O}_{\mathbb CP^1}(k)$ for $k \in \mathbb Z$.
In this case $\mathrm{Pic}_\g(\mathbb CP^1) =\mathbb C \times \mathbb Z$, where
$\mathbb Z=\op{Pic}(\mathbb CP^1)$. However, not all $\g$-equivariant line bundles are of the form $[D]$
for some $\g$-invariant divisor $D$. Instead, as a consequence of the necessary criterion of Theorem \ref{Th2},
we have $\mathrm{Div}_\g(\mathbb CP^1) = \mathbb Z$. For more details, see Example \ref{ex:aff1}.

In Section \ref{sect:bundles} we focus on the important cases of projectable Lie algebras of vector fields 
on affine bundles and on jet bundles. In these situations one can consider divisors whose restriction to fibers 
are polynomial. Let $\hat\g$ be a projectable Lie algebra of vector fields on the total space of an affine bundle 
$\pi\colon E\to M$ that preserves the affine structure on $E$, and let $\g=d\pi(\hat\g)\subset \vf(M)$.

 \begin{theorem}\label{Th3}
If $D$ is a $\hat \g$-invariant polynomial divisor on the affine bundle $E$, then $[D] = \pi^*L$ for some $\g$-equivariant line bundle
$L\in \Phi_1(\op{Pic}_\g(M))$.
 \end{theorem}

In other words, the $\hat\g$-equivariant line bundle over $E$ corresponding to a $\hat \g$-invariant polynomial divisor is the pullback of a $\g$-equivariant line bundle over $M$. The same idea works for jet bundles
because the bundle $\pi_{k+1,k}\colon \J^{k+1}\to\J^k$ for $k\geq1$ has a natural affine structure in fibers.
(For jet spaces of sections of line bundles with the contact transformation algebra, the natural affine structure 
in fibers starts at $k=2$, with the corresponding modification of the claim.)
 % The following is not a direct corollary of Theorem \ref{Th3} but elaborates on the same idea.

 \begin{theorem}\label{Th4}
 Let $\g^{(k)} \subset \vf(\J^k)$ be the prolongation of a Lie algebra $\g$ of point transformations on $\J^0$,
$0<k\leq\infty$. If $D$ is a $\g^{(k)}$-invariant divisor that is polynomial in fibers of 
$\pi_{k,1}\colon \J^k \to \J^1$, then $[D] = \pi_{k,1}^*L$ for some $\g^{(1)}$-equivariant line bundle 
 $L \in \Phi_1(\mathrm{Pic}_{\g^{(1)}}(\J^1))$.
 \end{theorem}

This result provides our main application for classification of global relative invariants 
of the prolonged $\g$-action on $\J^\infty$,  which is an essential step in the classification of all invariant differential equations  (see \cite{KS} for a series of examples of this technique).
We note that while the Gelfand-Fuks type cohomology 
$H^1(\g^{(\infty)},\mathcal{F}(\J^\infty))$ may be large and hard to compute, 
the theorem reduces the problem to finite dimensions.
To illustrate this, we will show how this allows to effectively treat relative differential invariants 
of curves in $\mathbb CP^2$ under the action of the M\"obius algebra of projective transformations 
 % $\mathfrak{sl}(3,\mathbb C)$
as well as relative differential invariants of second-order ODEs under the infinite-dimensional Lie algebra 
of point transformations.

The main results are proved and expanded in the following sections. To be precise, Theorem~\ref{Th1} corresponds to Propositions \ref{prop:HtotPic} and \ref{prop:TgM}, Theorem \ref{Th2} to Propositions \ref{prop:tobundle} and \ref{prop:orbitdim}, Theorem \ref{Th3} to Propositions \ref{prop:13} and \ref{prop:aff}, and Theorem \ref{Th4} is an instance of results summarized in Propositions \ref{prop:15}, \ref{prop:polynomialonjets} and \ref{prop:affjet}. Other results are presented in the main text, in particular Theorem \ref{th:orbitdimR} which is a partial converse to Theorem \ref{Th2} in the case of algebraic group actions. We end  with examples that illustrate computations of global relative differential invariants using our formalism.

In this paper we concentrate on the complex analytic and complex algebraic situation, using notation 
$\C P^n$ instead of $\mathbb P^n$ to stress that a part of our results extend to the real analytic and 
real algebraic case, with examples like real projective spaces $\R P^n$, real jet spaces $\J^\infty$, etc. 
In particular, examples A-C may be treated in the real context.

\section{Analytic invariant divisors and equivariant line bundles}\label{sect:divisorsandbundles}

Let $\g \subset \vf(M)$ denote a Lie algebra of holomorphic vector fields on the complex manifold $M$. 
 For $M= \mathbb C^n$ it is well-known that lifts of a Lie algebra $\g$ to the trivial line bundle $M \times \mathbb C$ are 
parametrized by the Chevalley-Eilenberg cohomology $\HH^1(\mathfrak g, \mathcal{O}(M))$. 
This is also the space, where weights of relative $\g$-invariants take values. We refer to \cite{CE,F} for 
the general Lie algebra cohomology theory, to \cite{COW} for its relation to relative (differential) invariants,
and to \cite{FO,S} for a relation to lifts.

The goal of this section is to generalize these results to arbitrary holomorphic line bundles over complex manifolds, and replace the notion of relative $\g$-invariant functions with $\g$-invariant divisors on $M$. 

\subsection{Picard group and multipliers}\label{sect:2.1}

Let us start with a quick overview of holomorphic line bundles, sufficient for our purpose (see \cite{GH,H}).  
For an open subset $U \subset M$ denote by $\mathcal{O}(U)$ the space of holomorphic functions on $U$, 
and by $\mathcal{O}^\times(U)$ the subspace of nonvanishing functions. 
The corresponding sheaves on $M$ are denoted by $\mathcal{O}$ and $\mathcal{O}^\times$, respectively.  
Let $\pi \colon L \to M$ be a line bundle and consider an open cover 
$\mathcal{U} = \{U_\alpha\}$ of $M$ that trivializes $\pi$, i.e., 
$\pi^{-1}(U_\alpha) \simeq U_\alpha \times \mathbb C$. The line bundle is uniquely determined 
by its transition functions $g_{\alpha\beta} \in \mathcal{O}^\times (U_\alpha \cap U_\beta)$, which
satisfy $ g_{\alpha \beta} g_{\beta \gamma} = g_{\alpha \gamma}$.
Two collections of transition functions $\{g_{\alpha \beta}\}$, $\{\tilde g_{\alpha \beta}\}$ 
define the same bundle if and only if $\tilde g_{\alpha \beta} = \frac{f_\alpha}{f_\beta} g_{\alpha \beta}$ for 
some functions $f_\alpha \in \mathcal{O}^\times(U_\alpha)$. 

This leads to a description of line bundles in terms of \v{C}ech cohomology. Define the complex
 \[ 
0 \xrightarrow{} \prod_{\alpha} \mathcal{O}^\times (U_\alpha) \xrightarrow{\delta^0} \prod_{\alpha \neq \beta} \mathcal{O}^\times(U_\alpha \cap U_\beta) \xrightarrow{\delta^1} \prod_{\alpha \neq \beta \neq \gamma \neq \alpha}\mathcal{O}^\times(U_\alpha \cap U_\beta \cap U_\gamma) \xrightarrow{} \cdots,
 \]
with differentials given by
 \begin{equation*}
(\delta^q \mu)_{\alpha_0 \cdots \alpha_{q+1}} = \prod_{i=0}^{q+1} \mu^{(-1)^{i+1}}_{\alpha_0 \cdots \hat \alpha_i \cdots \alpha_{q+1}} \Big|_{U_{\alpha_0} \cap \cdots \cap U_{\alpha_{q+1}}}, \qquad \mu = \{\mu_{\alpha_0 \cdots \alpha_q}\} \in \prod_{\alpha_0,\dots, \alpha_q} \mathcal{O}^\times(U_{\alpha_0} \cap \cdots \cap U_{\alpha_q}).
 \end{equation*}
In particular, $\delta^0$ and $\delta^1$ are defined in the following way:
 \begin{align*}
	(\delta^0\mu)_{\alpha \beta} &= \mu_{\alpha}/\mu_{\beta}, \qquad && \mu = \{\mu_\alpha\} \in \prod_{\alpha} \mathcal{O}^\times (U_\alpha),\\
	(\delta^1\nu)_{\alpha \beta \gamma} &=\frac{ \nu_{\alpha \gamma}}{ \nu_{\alpha \beta}  \nu_{\beta \gamma}} , \qquad && \nu = \{\nu_{\alpha \beta}\} \in \prod_{\alpha \neq \beta} \mathcal{O}^\times(U_\alpha \cap U_\beta).
 \end{align*}
The first \v{C}ech-cohomology with respect to the fixed open cover $\mathcal{U}$, 
defined by $\Hcech^1(\mathcal{U}, \mathcal{O}^\times) = \ker(\delta^1)/\mathrm{im}(\delta^0)$,
is the group of transition functions on $\mathcal{U}$ modulo the above equivalence relation. 

The Picard group $\mathrm{Pic}(M)$ of equivalence classes of holomorphic line bundles over $M$ can be 
described in terms of this cohomology group as follows:
 \begin{itemize}
\item If all line bundles are trivializable on the open charts in $\mathcal{U}$ (for instance, each $U_\alpha$ 
is biholomorphic to a polydisc with a possible factor $\mathbb C^\times$) 
then $\mathrm{Pic}(M) \simeq \Hcech^1(\mathcal{U},\mathcal{O}^\times)$. 
\item In general $\mathrm{Pic}(M) \simeq \Hcech^1(M,\mathcal{O}^\times):= 
\varinjlim \Hcech^1(\mathcal{U}, \mathcal{O}^\times)$ is the direct limit as $\mathcal{U}$ becomes finer.
 \end{itemize}
In both cases, the identification is a group isomorphism.  In particular, if the conditions of Leray's theorem hold, the first description is applicable (see \cite[p.40]{GH} or the simpler Theorem 12.8 of \cite{Forster}, which will usually be sufficient for us). 
 %We will mostly use the first approach in this paper by choosing a nice open cover $\mathcal{U}$, which we adapt in what follows.

%  \begin{remark}
% The differential $\delta^q$ is slightly unconventional, as our definition of $(\delta^q \mu)_{\alpha_0 \cdots 
% \alpha_{q+1}}$ is the multiplicative inverse of what one finds in most other sources.
%  \end{remark}

 \begin{definition}
A lift of $\mathfrak g \subset \vf(M)$ to the line bundle $\pi \colon L \to M$ is a Lie algebra 
$\hat{\mathfrak g} \subset \vf_{\text{proj}}(L)$ of projectable vector fields, such that 
$d\pi\colon\hat{\mathfrak g}\to\mathfrak g$ is a Lie algebra isomorphism and $\hat{\g}$ commutes 
with the natural vertical vector field $u \partial_u$ ($u$ is a linear fiber coordinate).  
The pair $(\pi, \hat \g)$ is called a $\g$-equivariant line bundle.
(We also refer to $\pi$ or $L$ as a $\g$-equivariant bundle when a lift exists.)
 \end{definition}

% \begin{example}
For instance, the canonical line bundle $K_M = \Lambda^{\dim M} T^*M$ (see \cite[Ch. 2.2]{H}) always 
admits a canonical lift of $\g \subset \vf(M)$. Thus it is an (often nontrivial) $\g$-equivariant line bundle.
% \end{example}

In general, the lift of a vector field $X\in\g$ can be defined on $\pi^{-1}(U_\alpha)\simeq U_\alpha\times\C$ by  
 \[
\hat X|_{U_\alpha} = X|_{U_\alpha} + \lambda_\alpha(X) u\partial_u,\quad 
\lambda_\alpha \in \g^* \otimes \mathcal{O}(U_\alpha),
 \] 
similar to formula (4.1) in \cite{FO}. To simplify notation, we will write $X$ instead of $X|_{U_\alpha}$ 
when there is no room for confusion. The condition $[\hat X, \hat Y] = \widehat{[X,Y]}$ for each $X,Y\in\g$ 
implies that $\lambda_\alpha$ satisfies
 \begin{equation}
X(\lambda_\alpha(Y))-Y(\lambda_\alpha(X))= \lambda_{\alpha}([X,Y]), \qquad \forall X \in \mathfrak g. \label{eq:cocycle1}
 \end{equation}
Changing the coordinate function on the fiber, $v=e^{\mu_\alpha}u$ for some function 
$\mu_\alpha \in \mathcal{O}(U_\alpha)$, gives
 \begin{equation*}
X+\lambda_\alpha(X)\,u \partial_u = X+(\lambda_\alpha(X)+X(\mu_\alpha))\,v \partial_v.
 \end{equation*}
In this sense, two lifts $\lambda_\alpha, \tilde \lambda_\alpha$ on $U_\alpha$ are equivalent if and only if there exists a $\mu_\alpha$ satisfying
 \begin{equation}
\tilde \lambda_\alpha(X) = \lambda_\alpha(X) + X( \mu_\alpha), \qquad  \forall X \in \mathfrak g. \label{eq:coboundary1}
 \end{equation}

The conditions \eqref{eq:cocycle1} and \eqref{eq:coboundary1} can be interpreted in terms of Lie algebra cohomology of $\mathfrak g$ with coefficients in the $\mathfrak g$-module $\mathcal{O}(U_\alpha)$. Consider the Chevalley-Eilenberg complex
 \[ 
0 \xrightarrow{} \mathcal{O}(U_\alpha)  \xrightarrow{d^0} \g^*\otimes \mathcal{O}(U_\alpha) \xrightarrow{d^1} \Lambda^2 \g^*\otimes \mathcal{O}(U_\alpha) \xrightarrow{} \cdots 
 \]
where the maps $d^0$ and $d^1$ are given by
 \begin{alignat*}{2}
(d^0 \mu_\alpha)(X) &= X(\mu_\alpha), \qquad &&\mu_\alpha \in \mathcal{O}(U_\alpha), \\
(d^1 \lambda_\alpha)(X,Y) &= X(\lambda_\alpha(Y))-Y(\lambda_\alpha(X))-\lambda_\alpha([X,Y]), \qquad  &&\lambda_\alpha  \in \g^* \otimes \mathcal{O}(U_\alpha),
 \end{alignat*}
for $X,Y \in \g$ (see \cite{CE}). Notice that $\mathrm{Hom}(\g,F)=\g^* \otimes F$ when one of the factors is finite-dimensional. If both factors are infinte-dimensional, a completion of the tensor product is required. We omit this from the notation, understanding by default that $\Lambda^i \g^*\otimes F$ may stand for $\mathrm{Hom}(\Lambda^i \g^*,F)$ here and below.

Define the cohomology groups
 \[
\HH^0(\mathfrak g,\mathcal{O}(U_\alpha)) = \ker(d^0), \qquad \HH^i(\mathfrak g, \mathcal{O}(U_\alpha)) = \ker(d^i)/\mathrm{im}(d^{i-1}), \quad i>0.
 \]
It is clear that $\lambda_\alpha \in \mathfrak g^* \otimes \mathcal{O}(U_\alpha)$ defines a lift of $\mathfrak g$ 
to $U_\alpha \times \mathbb C$ if and only if $d^1 \lambda_\alpha=0$. Furthermore, two cocycles 
$\lambda_\alpha, \tilde \lambda_\alpha$ define equivalent lifts if and only if 
$\tilde \lambda_\alpha = \lambda_\alpha + d^0 \mu_\alpha$ for some $\mu_\alpha \in \mathcal{O}(U_\alpha)$. 
Thus, equivalence classes of lifts of $\mathfrak g|_{U_\alpha}$ to $U_\alpha \times \mathbb C$ are in 
one-to-one correspondence with elements in $\HH^1(\mathfrak g, \mathcal{O}(U_\alpha))$. 
(Note also that $\HH^0(\g, \mathcal{O}(U_\alpha))=\mathcal{O}(U_\alpha)^\g$ consists of $\g$-invariants.)

\begin{remark} \label{rk:complex}
If $U_\alpha$ is a polydisc for each $\alpha$, then any function in $\mathcal{O}^\times (U_\alpha)$ is of 
the form $e^\mu$, and the argument above works. If $U_\alpha$ is a general open set, one replaces 
$e^{\mu_\alpha} f_\alpha$ with $\mu_\alpha f_\alpha$, where $\mu_\alpha \in \mathcal{O}^\times(U_\alpha)$. 
Then the local lifts are in one-to-one correspondence with elements in the cohomology group of the complex
 \begin{equation}\label{CEg}
0 \xrightarrow{} \mathcal{O}^\times (U_\alpha)  \xrightarrow{d^0 \log} \g^*\otimes \mathcal{O}(U_\alpha) \xrightarrow{d^1} \Lambda^2 \g^*\otimes \mathcal{O}(U_\alpha) \xrightarrow{} \cdots. 
 \end{equation}
We will use this slightly modified complex below with the notation 
 \[
\Hmod^1(\g,\mathcal{O}(U_\alpha)) = \frac{\ker(d^1)}{\mathrm{im}(d^0 \log)}.
 \]
\end{remark}

%Denote by $\mathfrak{M}_\g(\mathcal{U}):=\{([\lambda_\alpha])\in\prod_\alpha\Hmod^1(\g,\mathcal{O}(U_\alpha)) \,|\, [\lambda_\alpha|_{U_\alpha\cap U_\beta}]=[\lambda_\beta|_{U_\alpha\cap U_\beta}]\}$ the collection 
%of local (infinitesimal) multipliers of $\g$ that are equivalent on overlaps. Similar to the Picard group, we introduce:

% \begin{definition}
%The multipliers of $\g$ on $M$ is the direct limit by refinements % of covers
%$\mathfrak{M}_\g(M):=\varinjlim\mathfrak{M}_\g(\mathcal{U})$.
% \end{definition}
 
%Note that the global multipliers $\Hmod^1(\g,\mathcal{O}(M))$ may be trivial as the algebra of global
%functions $\mathcal{O}(M)$ may be small (for instance $\C$ for compact $M$).

Elements in % $\mathfrak{M}_\g(\mathcal{U})$ 
$\Hmod^1(\g,\mathcal{O}(U_\alpha))$ yield local lifts of $\g$ to $\pi^{-1}(U_\alpha)$ that may not glue together to a global lift on $L$. 
 % Furthermore, a nontrivial lift on $L$ may restrict to a trivial lift on $\pi^{-1}(U_\alpha)$ for some $\alpha$. 
On $U_\alpha \cap U_\beta$ a lift is given by both $X+\lambda_\alpha(X)\,u_\alpha \partial_{u_\alpha}$ 
and $X+\lambda_\beta(X)\,u_\beta\partial_{u_\beta}$. The fiber coordinates relate on overlaps by
$u_\alpha = g_{\alpha\beta} u_\beta$, where the transition functions $\{g_{\alpha \beta}\}$ 
represent an element of $\Hcech^1(\mathcal{U},\mathcal{O}^\times(M))$.  Thus 
$X+\lambda_\alpha(X)\,u_\alpha \partial_{u_\alpha}$ becomes 
$X+(\lambda_\alpha(X)-X(g_{\alpha \beta})/g_{\alpha \beta})\,u_\beta \partial_{u_\beta}$, 
resulting in the following compatibility condition on $U_\alpha \cap U_\beta$:
 \begin{equation}\label{eq:compatibility}
\lambda_\alpha(X)-\lambda_\beta(X)=\frac{X(g_{\alpha \beta})}{g_{\alpha \beta}} = X(\log g_{\alpha \beta}), \qquad \forall X \in \mathfrak g.
 \end{equation}
 %; in other words $(\delta^0 \lambda)_{\alpha \beta}= d^0 \log g_{\alpha \beta}$.

\subsection{A double complex} \label{sect:double}

To better understand the compatibility condition, consider the double complex
\[\begin{tikzcd}
	{C^{0,0}} & {C^{1,0}} & {C^{2,0}} & {} \\
	{C^{0,1}} & {C^{1,1}} & {C^{2,1}} & {} \\
	{C^{0,2}} & {C^{1,2}} & {C^{2,2}} & {} \\
	{} & {} & {}
	\arrow["{\delta^{0,0}}"', from=1-1, to=2-1]
	\arrow["{\delta^{0,1}}"', from=2-1, to=3-1]
	\arrow["{d^{0,0}}", from=1-1, to=1-2]
	\arrow["{d^{1,0}}", from=1-2, to=1-3]
	\arrow["{d^{0,1}}", from=2-1, to=2-2]
	\arrow["{d^{1,1}}", from=2-2, to=2-3]
	\arrow["{d^{0,2}}", from=3-1, to=3-2]
	\arrow["{\delta^{1,0}}"', from=1-2, to=2-2]
	\arrow["{\delta^{1,1}}"', from=2-2, to=3-2]
	\arrow["{\delta^{2,0}}"', from=1-3, to=2-3]
	\arrow["{\delta^{0,2}}"', from=3-1, to=4-1]
	\arrow["{\delta^{1,2}}"', from=3-2, to=4-2]
	\arrow["{d^{1,2}}", from=3-2, to=3-3]
	\arrow["{\delta^{2,1}}"', from=2-3, to=3-3]
	\arrow["{\delta^{2,2}}"', from=3-3, to=4-3]
	\arrow["{d^{2,0}}", from=1-3, to=1-4]
	\arrow["{d^{2,1}}", from=2-3, to=2-4]
	\arrow["{d^{2,2}}", from=3-3, to=3-4]
\end{tikzcd}\]
where $C^{p,q}$ are given by
 \begin{align*}
C^{0,q} &= \prod_{\alpha_0, \cdots, \alpha_q} \mathcal{O}^\times (U_{\alpha_0} \cap \cdots\cap  U_{\alpha_q}),\\
C^{p,q} &= \prod_{\alpha_0, \cdots, \alpha_q} \Lambda^p \mathfrak g^* \otimes \mathcal{O}(U_{\alpha_0} \cap \cdots\cap  U_{\alpha_q}), \quad p \geq 1,
 \end{align*}
and the differentials $\delta^{p,q} \colon C^{p,q} \to C^{p,q+1}$ and $d^{p,q} \colon C^{p,q} \to C^{p+1,q}$ are defined for $p=0$ by
 \begin{align*}
(\delta^{0,q} \mu)_{\alpha_0 \cdots \alpha_{q+1}} &= \prod_{i=0}^{q+1} \mu_{\alpha_0 \cdots \hat \alpha_i \cdots \alpha_{q+1}}^{(-1)^{i+1}} \Big|_{U_{\alpha_0} \cap \cdots \cap U_{\alpha_{q+1}}}, \\
(d^{0,q}\mu_{\alpha_0 \cdots \alpha_q})(X) &= X(\log \mu_{\alpha_0 \cdots \alpha_q}) = \frac{X(\mu_{\alpha_0 \cdots \alpha_q})}{\mu_{\alpha_0 \cdots \alpha_q}},
 \end{align*}
and for $p > 0$ by
 \begin{align*}
(\delta^{p,q} \mu)_{\alpha_0 \cdots \alpha_{q+1}} &= \sum_{i=0}^{q+1} (-1)^{i+1} \mu_{\alpha_0 \cdots \hat \alpha_i \cdots \alpha_{q+1}} \Big|_{U_{\alpha_0} \cap \cdots \cap U_{\alpha_{q+1}}}, \\
(d^{p,q}\mu_{\alpha_0 \cdots \alpha_q})(X_0, \dots, X_p) &=\sum_{i=0}^{p} (-1)^i X_i(\mu_{\alpha_0 \cdots \alpha_q}(X_0,\dots, X_{i-1}, X_{i+1}, \dots, X_p))\\
+\sum_{i<j}(-1)^{i+j}&(\mu_{\alpha_0 \cdots \alpha_q}([X_i,X_j],X_0,\dots, X_{i-1}, X_{i+1}, \dots, X_{j-1}, X_{j+1}, \dots, X_p)).
 \end{align*}

We will sometimes write $C^{p,q}(\g,\mathcal{U})$ for precision when there would otherwise be ambiguity. 
The horizontal lines ($q$ fixed) are nearly Chevalley-Eilenberg complexes of $\mathfrak g$ with coefficients in the 
$\mathfrak g$-modules $\mathcal{O}(U_\alpha), \mathcal{O}(U_{\alpha} \cap U_{\beta})$, etc; 
however $(C^{0,q},d^{0,q})$ are adjusted in accordance with Remark \ref{rk:complex}. 
The vertical lines ($p$ fixed) are \v{C}ech complexes with respect to the open cover $\mathcal{U}$.

 \begin{remark}
For $C^{0,q}$ it is natural to use multiplicative notation (with identity element $1$) while for $C^{p,q}$ for $p \geq 1$ it is better to use additive notation (with identity element $0$). Using these notations consistently becomes difficult when we are dealing with this double complex, and even more so when we work with the total complex defined below. We will therefore use $0$ to denote the identity element in these groups, and in the corresponding cohomology groups.
 \end{remark}

The total complex corresponding to the double complex $C^{\bullet, \bullet}$ is defined as follows:
\[\mathrm{Tot}^r(C) = \prod_{p+q=r} C^{p,q}, \qquad \partial^r=\sum_{p+q=r} (d^{p,q}+(-1)^p\delta^{p,q})\colon \mathrm{Tot}^r(C) \to \mathrm{Tot}^{r+1}(C).\]
The identity $\partial^{i+1} \circ \partial^i =0$ expresses the fact that the double complex is a commutative diagram. The cohomology groups of the total complex are defined in the usual way:
\[\HH^0(\mathrm{Tot}^\bullet(C))= \ker(\partial^0), \qquad \HH^i(\mathrm{Tot}^\bullet(C))= \frac{\ker(\partial^i)}{\mathrm{im}(\partial^{i-1})}.\]

The double complex also gives us several complexes of cohomology groups. The cohomology groups with respect to $d^{i,j}$ (with $j$ fixed) make up the following complexes:
% https://q.uiver.app/#q=WzAsMTIsWzEsMCwiSF9kXjEoQ157MSwwfSkiXSxbMSwxLCJIX2ReMShDXnsxLDF9KSJdLFsxLDIsIkhfZF4xKENeezEsMn0pIl0sWzAsMCwiSF9kXjAoQ157MCwwfSkiXSxbMCwxLCJIX2ReMChDXnswLDF9KSJdLFswLDIsIkhfZF4wKENeezAsMn0pIl0sWzIsMCwiSF9kXjIoQ157MiwwfSkiXSxbMiwxLCJIX2ReMihDXnsyLDF9KSJdLFsyLDIsIkhfZF4yKENeezIsMn0pIl0sWzEsMywiXFxidWxsZXQiXSxbMCwzLCJcXGJ1bGxldCJdLFsyLDMsIlxcYnVsbGV0Il0sWzMsNCwiXFxkZWx0YV8qXnswLDB9Il0sWzQsNSwiXFxkZWx0YV8qXnswLDF9Il0sWzAsMSwiXFxkZWx0YV8qXnsxLDB9Il0sWzEsMiwiXFxkZWx0YV8qXnsxLDF9Il0sWzYsNywiXFxkZWx0YV8qXnsyLDB9Il0sWzcsOCwiXFxkZWx0YV8qXnsyLDF9Il0sWzIsOSwiXFxkZWx0YV8qXnsxLDJ9Il0sWzUsMTAsIlxcZGVsdGFfKl57MCwyfSJdLFs4LDExLCJcXGRlbHRhXypeezIsMn0iXV0=
\[\begin{tikzcd}
	{\HH_d^0(C^{\bullet,0})} & {\HH_d^1(C^{\bullet,0})} & {\HH_d^2(C^{\bullet,0})} \\
	{\HH_d^0(C^{\bullet,1})} & {\HH_d^1(C^{\bullet,1})} & {\HH_d^2(C^{\bullet,1})} \\
	{\HH_d^0(C^{\bullet,2})} & {\HH_d^1(C^{\bullet,2})} & {\HH_d^2(C^{\bullet,2})} \\
	{} & {} & {}
	\arrow["{\delta_*^{0,0}}", from=1-1, to=2-1]
	\arrow["{\delta_*^{0,1}}", from=2-1, to=3-1]
	\arrow["{\delta_*^{1,0}}", from=1-2, to=2-2]
	\arrow["{\delta_*^{1,1}}", from=2-2, to=3-2]
	\arrow["{\delta_*^{2,0}}", from=1-3, to=2-3]
	\arrow["{\delta_*^{2,1}}", from=2-3, to=3-3]
	\arrow["{\delta_*^{1,2}}", from=3-2, to=4-2]
	\arrow["{\delta_*^{0,2}}", from=3-1, to=4-1]
	\arrow["{\delta_*^{2,2}}", from=3-3, to=4-3]
\end{tikzcd}\]
Simultaneously, the cohomology groups with respect to $\delta^{i,j}$ (with $i$ fixed) also give complexes:
% https://q.uiver.app/#q=WzAsMTIsWzAsMCwiSF9cXGRlbHRhXjAoQ157MCwwfSkiXSxbMSwwLCJIX1xcZGVsdGFeMChDXnsxLDB9KSJdLFsyLDAsIkhfXFxkZWx0YV4wKENeezIsMH0pIl0sWzAsMSwiSF9cXGRlbHRhXjEoQ157MCwxfSkiXSxbMSwxLCJIX1xcZGVsdGFeMShDXnsxLDF9KSJdLFsyLDEsIkhfXFxkZWx0YV4xKENeezIsMX0pIl0sWzAsMiwiSF9cXGRlbHRhXjIoQ157MCwyfSkiXSxbMSwyLCJIX1xcZGVsdGFeMihDXnsxLDJ9KSJdLFsyLDIsIkhfXFxkZWx0YV4yKENeezIsMn0pIl0sWzMsMiwiXFxidWxsZXQiXSxbMywxLCJcXGJ1bGxldCJdLFszLDAsIlxcYnVsbGV0Il0sWzAsMSwiZF8qXnswLDB9Il0sWzEsMiwiZF8qXnsxLDB9Il0sWzMsNCwiZF8qXnswLDF9Il0sWzQsNSwiZF8qXnsxLDF9Il0sWzYsNywiZF8qXnswLDJ9Il0sWzcsOCwiZF8qXnsxLDJ9Il0sWzgsOSwiZF8qXnsyLDJ9Il0sWzUsMTAsImRfKl57MiwxfSJdLFsyLDExLCJkXypeezIsMH0iXV0=

\[\begin{tikzcd}
	{\HH_\delta^0(C^{0,\bullet})} & {\HH_\delta^0(C^{1,\bullet})} & {\HH_\delta^0(C^{2,\bullet})} & {} \\
	{\HH_\delta^1(C^{0,\bullet})} & {\HH_\delta^1(C^{1,\bullet})} & {\HH_\delta^1(C^{2,\bullet})} & {} \\
	{\HH_\delta^2(C^{0,\bullet})} & {\HH_\delta^2(C^{1,\bullet})} & {\HH_\delta^2(C^{2,\bullet})} & {}
	\arrow["{d_*^{0,0}}", from=1-1, to=1-2]
	\arrow["{d_*^{1,0}}", from=1-2, to=1-3]
	\arrow["{d_*^{0,1}}", from=2-1, to=2-2]
	\arrow["{d_*^{1,1}}", from=2-2, to=2-3]
	\arrow["{d_*^{0,2}}", from=3-1, to=3-2]
	\arrow["{d_*^{1,2}}", from=3-2, to=3-3]
	\arrow["{d_*^{2,2}}", from=3-3, to=3-4]
	\arrow["{d_*^{2,1}}", from=2-3, to=2-4]
	\arrow["{d_*^{2,0}}", from=1-3, to=1-4]
\end{tikzcd}\]
Thus we get the induced cohomology groups
 \[ 
\HH_\delta^i(\HH_d^j(C^{\bullet, \bullet})) = \frac{\ker(\delta^{j,i}_*)}{\mathrm{im}(\delta^{j,i-1}_*)}, \qquad \HH_d^i(\HH_\delta^j(C^{\bullet,\bullet})) = \frac{\ker(d^{i,j}_*)}{\mathrm{im}(d^{i-1,j}_*)},
 \]
for $i \geq 1$ and
 \[ 
\HH_\delta^0(\HH_d^j(C^{\bullet, \bullet})) = \ker(\delta_*^{j,0}), \qquad \HH_d^0(\HH_\delta^j(C^{\bullet,\bullet})) = \ker(d_*^{0,j}).
 \]

In the general setting of the total complex, we have the two projections (homomorphisms)
 \[\begin{tikzcd}
& {\HH^1(\mathrm{Tot}^\bullet(C))} \\
{\HH_\delta^1(C^{0,\bullet})} && {\HH_d^1(C^{\bullet,0})}
\arrow["{\Phi_1}", from=1-2, to=2-1]
\arrow["{\Phi_2}"', from=1-2, to=2-3]
 \end{tikzcd}\]
defined by
 \[ 
\Phi_1([(g,\lambda)]) = [g], \qquad \Phi_2([(g,\lambda)]) = [\lambda],
 \]
where $[(g,\lambda)]$ denotes the equivalence class of $(g,\lambda) \in \ker(\partial^1)$.   
We use the notation $g = \{g_{\alpha \beta}\}$ for an element in $C^{0,1}$ and 
$\lambda=\{\lambda_\alpha\}$ for an element in $C^{1,0}$, and note that 
%for a good cover $\mathcal{U}$:
% \[
%\HH_{\delta}^1(C^{0,\bullet})=\check\HH^1(\mathcal{U},\mathcal{O}^\times) \simeq \mathrm{Pic}(M), \qquad 
%\HH_{d}^1(C^{\bullet,0})= \prod_\alpha \Hmod^1(\mathfrak g,\mathcal{O}(U_\alpha))\supset\mathfrak{M}_\g(M).
 %\]
%Note that we have $\mathfrak{M}_\g(\mathcal{U})=\HH_\delta^0(\HH_d^1(C^{\bullet, \bullet}))$,
%which also equals $\mathfrak{M}_\g(M)$ for a good cover $\mathcal{U}$.
 \[
\HH_{\delta}^1(C^{0,\bullet})=\Hcech^1(\mathcal{U},\mathcal{O}^\times), \qquad 
\HH_{d}^1(C^{\bullet,0})= \prod_\alpha \Hmod^1(\mathfrak g,\mathcal{O}(U_\alpha)).
 \]
 \begin{lemma} \label{lem:proj}
The maps $\Phi_1, \Phi_2$ have the following kernels:
 \[ 
\ker(\Phi_1)  \simeq \HH_d^1(\HH_\delta^0(C^{\bullet,\bullet})), \qquad 
\ker(\Phi_2) \simeq \HH_\delta^1(\HH_d^0(C^{\bullet,\bullet})).
 \]
 \end{lemma}

 \begin{proof}
The arguments for the two isomorphisms are similar to each other, so we prove the statement only for $\ker(\Phi_2)$. 
If $[(g,\lambda)] \in \ker(\Phi_2)$, then there exists an element $\tilde g \in \ker(\delta^{0,1})$ such that 
$[(\tilde g,0)] = [(g,\lambda)]$. We have $d^{0,1}\tilde g = \delta^{1,0} 0=0$, so that $\tilde g\in\ker (d^{0,1})$. 
Furthermore, since $(\tilde g,0) + \partial^0(\mu) = (\tilde g \cdot \delta^{0,0} \mu, d^{0,0}\mu)$, 
the freedom in choice of representative $\tilde g$ is exactly $\delta^{0,0}(\ker(d^{0,0}))$. 
Thus $\ker(\Phi_2)=\HH_\delta^1(\HH_d^0(C^{\bullet,\bullet}))$.
 \end{proof}

Lemma \ref{lem:proj} is equivalent to exactness of the two sequences
 \begin{gather}
0 \longrightarrow \HH_d^1(\HH_\delta^0(C^{\bullet,\bullet})) \longrightarrow \HH^1(\mathrm{Tot}^\bullet(C)) \longrightarrow \mathrm{im}(\Phi_1) \longrightarrow 0 \label{ll1},\\
0 \longrightarrow \HH_\delta^1(\HH_d^0(C^{\bullet,\bullet})) \longrightarrow \HH^1(\mathrm{Tot}^\bullet(C)) \longrightarrow \mathrm{im}(\Phi_2) \longrightarrow 0 \label{ll2},
 \end{gather}
 and, furthermore, we have
 \[
%\HH_\delta^1(\HH_d^0(C^{\bullet,\bullet})) = \HH^1(M, (\mathcal{O}^\times)^\g), \qquad \HH_d^1(\HH_\delta^0(C^{\bullet,\bullet})) = \Hmod^1(\g, \mathcal{O}(M)),
 \varinjlim \HH_\delta^1(\HH_d^0(C^{\bullet,\bullet})) = \Hcech^1(M, (\mathcal{O}^\times)^\g), \qquad \varinjlim \HH_d^1(\HH_\delta^0(C^{\bullet,\bullet})) = \Hmod^1(\g, \mathcal{O}(M)),
 \]
where $(\mathcal{O}^\times)^\g \subset \mathcal{O}^\times$ is the subsheaf of $\g$-invariants. 

% \begin{cor}\label{cor:H1trivial}
%{\rm (i)} If  $\mathrm{Pic}(M)=0$, then $\HH^1(\mathrm{Tot}^\bullet (C)) \simeq \HH_d^1(\HH_\delta^0(C^{\bullet, \bullet}))$.\\
%{\rm (ii)} Likewise, if $\mathfrak{M}_\g(M)=0$, then $\HH^1(\mathrm{Tot}^\bullet (C)) \simeq \HH_\delta^1(\HH_d^0(C^{\bullet,\bullet}))$.
% \end{cor}
  \begin{cor}\label{cor:H1trivial}
{\rm (i)} If  $\HH_\delta^1(C^{0,\bullet}) = 0$, then $\HH^1(\mathrm{Tot}^\bullet (C)) \simeq \HH_d^1(\HH_\delta^0(C^{\bullet, \bullet}))$.\\
{\rm (ii)} Likewise, if $\HH_d^1(C^{\bullet,0})=0$, then $\HH^1(\mathrm{Tot}^\bullet (C)) \simeq \HH_\delta^1(\HH_d^0(C^{\bullet,\bullet}))$.
 \end{cor}

 \begin{lemma}\label{lem:impsi}
The images of $\Phi_1,\Phi_2$ are given by the following exact sequences:
 \begin{align*}
0 \longrightarrow \mathrm{im}(\Phi_1) \longrightarrow \HH_d^0(\HH_\delta^1(C^{\bullet, \bullet})) 
\longrightarrow \HH_d^2(\HH_\delta^0(C^{\bullet,\bullet})), \\
0 \longrightarrow \mathrm{im}(\Phi_2) \longrightarrow \HH_\delta^0(\HH_d^1(C^{\bullet,\bullet})) 
\longrightarrow \HH_\delta^2(\HH_d^0(C^{\bullet,\bullet})).
 \end{align*}
 \end{lemma}

 \begin{proof}
We give the proof for the first exact sequence. The proof for the second one is similar.  Consider an element $[g] \in \mathrm{im}(\Phi_1) \subset \HH_\delta^1(C^{0,\bullet})$. Since it lies in the image of $\Phi_1$, there exists an element $\lambda \in C^{1,0}$ satisfying $\delta^{1,0} \lambda = d^{0,1} g$. This implies that $d_*^{0,1}[g] = [d^{0,1} g] = [\delta^{1,0} \lambda]=0$, and thus $[g] \in \HH_d^0(\HH_\delta^1(C^{\bullet,\bullet}))$. The map $\HH_\delta^1(C^{0,\bullet}) \supset \mathrm{im}(\Phi_1) \to  \HH_d^0(\HH_\delta^1(C^{\bullet,\bullet}))$ is obviously injective.
	
Now, consider an element $[g] \in \HH_d^0(\HH_\delta^1(C^{\bullet,\bullet}))$. Since $d_*^{0,1}[g]=0$, there exists an element $\lambda \in C^{1,0}$ satisfying $\delta^{1,0} \lambda = d^{0,1} g$. If $\tilde \lambda \in C^{1,0}$ is another such element, then $\tilde \lambda - \lambda=\lambda_0 \in \HH_\delta^0(C^{1,\bullet})$. The element $d^{1,0}(\lambda + \lambda_0) \in C^{2,0}$ is $\delta^{2,0}$-closed since $\delta^{2,0} \circ d^{1,0} = d^{1,1} \circ \delta^{1,0}$ and $\delta^{1,0} \lambda=d^{0,1} g$. Thus $d^{1,0}(\lambda + \lambda_0) \in \HH_\delta^0(C^{2,\bullet})$. Since the freedom in representative $\lambda + \lambda_0$ is exactly $\HH_\delta^0(C^{1,\bullet})$, we obtain a unique element $[d^{1,0} \lambda] \in \HH^2_d(\HH_\delta^0(C^{\bullet,\bullet}))$. We have $[d^{1,0} \lambda]=0$, or equivalently $d^{1,0} \lambda \in d_*^{1,0}(\HH_\delta^0(C^{1,\bullet}))$, if and only if $[g] \in \mathrm{im}(\Phi_1)$.
 \end{proof}
While $\mathrm{Pic}(M) = \varinjlim \HH_\delta^1(C^{0,\bullet})$ plays an important role, the group $\HH_d^1(C^{\bullet,0})$ will in general grow without bound as the cover $\mathcal{U}$ becomes finer.  Therefore, as a counterpart to $\Hcech^1(\mathcal{U},\mathcal{O}^\times)= \HH_\delta^1(C^{0,\bullet})$, we define $\mathfrak{M}_\g(\mathcal{U}) := \HH_\delta^0(\HH_d^1(C^{\bullet, \bullet}))$ which can be interpreted as the collection of local (infinitesimal) multipliers of $\g$ with respect to the cover $\mathcal{U}$ that are equivalent on overlaps. Note that this is also a reasonable definition in this context due to Lemma \ref{lem:impsi}.
 \begin{definition}
 The group of multipliers of $\g$ on $M$ is the direct limit $\mathfrak{M}_\g(M) := \varinjlim \mathfrak{M}_\g(\mathcal{U})$.
 \end{definition}
The group $\mathfrak{M}_\g(M)$ should not be confused with the group of global multipliers $\Hmod^1(\g,\mathcal{O}(M))$, which is often trivial as the algebra of global functions $\mathcal{O}(M)$ may be small (for instance $\C$ for compact $M$). To see the difference, consider an element $\lambda = \{\lambda_\alpha\} \in \ker(d^{1,0})$. We have $[\lambda]  \in \HH_d^1(\HH_\delta^0(C^{\bullet,\bullet}))$ if and only if $\lambda_\alpha = \lambda_\beta$ on $U_\alpha \cap U_\beta$, and $[\lambda] \in \mathfrak{M}_\g(\{U_\alpha\})=\HH_\delta^0(\HH_d^1(C^{\bullet,\bullet}))$ if and only if $\lambda_\alpha = \lambda_\beta + d \log \mu_{\alpha \beta}$ for some $\mu=\{\mu_{\alpha \beta}\} \in \mathcal{O}^\times(U_\alpha \times U_\beta)$.

\subsection{The equivariant Picard group}\label{sect:Picg}

From the description of lifts at the end of Section \ref{sect:2.1} we see that 
the pair $(g,\lambda)\in C^{0,1}\times C^{1,0}$ defines a $\g$-equivariant line bundle if and only if 
 \[
\delta^{0,1} g=0,\quad d^{1,0}\lambda=0,\quad d^{0,1}g=\delta^{1,0}\lambda
\quad\Leftrightarrow\quad (g,\lambda)\in\ker(\partial^1).
 \] 
The three conditions correspond to the cocycle condition for transition functions, the cocycle condition 
for the local lift \eqref{eq:cocycle1} and the compatibility condition \eqref{eq:compatibility}, respectively.  
Rescaling the fiber coordinates $u_\alpha$ in the line bundle corresponds exactly to changing the cocycle $(g,\lambda)$ by a coboundary in $\mathrm{im}(\partial^0)$.
%Changing cocycles $g$ and $\lambda$ by coboundaries is due to a rescaling of the coordinate $u$ in the line bundle,
%so these choices are entangled and correspond to $(g,\lambda)\in\mathrm{im}(\partial^0)$. Thus we are lead to

 \begin{definition}
The group of equivalence classes of $\g$-equivariant line bundles is called the $\g$-equivariant Picard group 
and denoted by  $\mathrm{Pic}_\g(M):=\varinjlim\HH^1(\mathrm{Tot}^\bullet(C))$, where we exploit
the direct limit by refinements (or use a fine cover $\mathcal{U}$) as before.
 \end{definition}
Denoting by $\mathfrak{C}_\g$ the modified Chevalley-Eilenberg sheaf complex \eqref{CEg},
$\mathrm{Pic}_\g(M)$ may be identified with the first hypercohomology $\mathbb{H}^1(M,\mathfrak{C}_\g)$,
cf.\ \cite[Ch.\,3.5]{GH} for a discussion of hypercohomology $\mathbb{H}^q$.

The maps $\Phi_1 \colon \HH^1(\mathrm{Tot}^\bullet(C)) \to \HH_\delta^1(C^{0,\bullet})$ and $\Phi_2 \colon \HH^1(\mathrm{Tot}^\bullet(C)) \to \HH_d^1(C^{\bullet,0})$ induce maps 
\[\Phi_1 \colon \mathrm{Pic}_\g(M) \to \mathrm{Pic}(M), \qquad \Phi_2 \colon \mathrm{Pic}_\g(M) \to \mathfrak{M}_\g(M),\]
denoted by the same letters (Lemma \ref{lem:impsi} justifies the choice of codomain for the second map).
We define $\mathrm{Pic}^{\op{red}}_\g(M):=\mathrm{im}(\Phi_1\times\Phi_2)\subset
\op{Pic}(M)\times\mathfrak{M}_\g(M)$, which we call the reduced $\g$-equivariant Picard group, and denote by $\Psi_1,\Psi_2$ the projections 
\[\Psi_1 \colon \mathrm{Pic}^{\op{red}}_\g(M) \to \mathrm{Pic}(M), \qquad \Psi_2 \colon \mathrm{Pic}_\g(M) \to \mathfrak{M}_\g(M).\] 
Then $\varpi:=\Phi_1\times\Phi_2$ epimorphically maps $\mathrm{Pic}_\g(M)$ to
$\mathrm{Pic}^{\op{red}}_\g(M)$.

 \begin{prop} \label{prop:equivariantbundle}
$\mathrm{Pic}^{\op{red}}_\g(M)  \simeq \ker(\partial^1)/\sim$, where the equivalence relation is defined by
$(g,\lambda)\sim(\tilde g, \tilde \lambda)$ if $\tilde g=g \cdot \delta^{0,0} \nu$ and 
$\tilde \lambda = \lambda + d^{0,0} \mu$, where $\mu,\nu \in C^{0,0}$ 
satisfy
 \[
\mu/\nu \in \ker (d^{0,1} \circ \delta^{0,0})=\ker(\delta^{1,0} \circ d^{0,0}).
 \]
 \end{prop}

 \begin{proof}
The reduced equivalence relation is weaker, as the coboundaries for $g$ and 
$\lambda$ can be chosen independently. Thus if $(g,\lambda)\in\ker(\partial^1)$ and 
$(g\cdot \delta^{0,0} \nu, \lambda+d^{0,0} \mu) \in C^{0,1} \times C^{1,0}$
is an equivalent cocycle then it automatically satisfies the first two conditions: 
$\delta^{0,1}(g\cdot\delta^{0,0}\nu)=0$ and $d^{1,0}(\lambda+d^{0,0}\mu)=0$. 
However the third condition applied to the new pair is 
$d^{0,1}(g\cdot\delta^{0,0}\nu)=\delta^{1,0} (\lambda+d^{0,0}\mu)$, which is equivalent to 
$\mu/\nu\in\ker(d^{0,1}\circ\delta^{0,0})=\ker(\delta^{1,0}\circ d^{0,0})$.
 \end{proof}

Let us investigate the relationship between $\mathrm{Pic}_\g(M)$ and $\mathrm{Pic}^{\op{red}}_\g(M)$. 
By Proposition \ref{prop:equivariantbundle} the admissible pair $(\nu,\mu)\in C^{0,0}\times C^{0,0}$ 
characterizing the freedom in choice of representatives $(g,\lambda)  \in \ker(\partial^1)$ for the reduced group can be rewritten as 
$(\nu,\mu)=(\nu/\mu,1)\cdot(\mu,\mu)\simeq(\nu/\mu,\mu) \in \ker(d^{0,1}\circ\delta^{0,0})\times C^{0,0}$. 
It follows that  (for a good cover) $\mathrm{Pic}^{\op{red}}_\g(M)$ is equal to
 \[
 \frac{\ker(\partial^1)}{\mathrm{im}(\tilde\partial)\cdot\mathrm{im}(\partial^0)} = 
\frac{\ker(\partial^1)}{\frac{\mathrm{im}(\tilde \partial)}{\mathrm{im}(\tilde \partial)\cap\mathrm{im}(\partial^0)}
\cdot\mathrm{im}(\partial^0)}= \frac{\HH^1(\mathrm{Tot}^\bullet (C))}{\frac{\mathrm{im}(\tilde \partial)}{\mathrm{im}(\tilde \partial) \cap \mathrm{im}(\partial^0)}},
 \]
where the map $\tilde\partial \colon \ker(d^{0,1}\circ\delta^{0,0})\to C^{0,1}\times C^{1,0}$ is defined 
as $\delta^{0,0}\times0$. We have:
 \begin{align*}
\mathrm{im}(\partial^0) &= \{(\delta^{0,0} \mu, d^{0,0} \mu)\in C^{0,1}\times C^{1,0} \mid\mu\in C^{0,0}\}, \\ 
\mathrm{im}(\tilde \partial) &= \{(\delta^{0,0} \kappa, 0) \in C^{0,1} \times C^{1,0} \mid \kappa \in \ker(d^{0,1} \circ \delta^{0,0})\} \simeq \ker(d^{0,1} \circ \delta^{0,0}) /\ker(\delta^{0,0}) .
 \end{align*}
If $d^{0,0} \mu =0$, then $\mu \in\ker(\delta^{1,0} \circ d^{0,0}) =  \ker(d^{0,1} \circ \delta^{0,0})$, 
and therefore
 \begin{align*}
\mathrm{im}(\tilde\partial)\cap\mathrm{im}(\partial^0) &= 
\{(\delta^{0,0} \mu,0) \in C^{0,1} \times C^{1,0} \mid \mu \in \ker (d^{0,0}) \} \\
&= \delta^{0,0}(\ker(d^{0,0})) \simeq \ker(d^{0,0})/(\ker(\delta^{0,0}) \cap \ker(d^{0,0})).
 \end{align*}
It follows that 
 \begin{equation*}\label{TgM}
\frac{\mathrm{im}(\tilde \partial)}{\mathrm{im}(\tilde \partial) \cap \mathrm{im}(\partial^0)} = \frac{\ker(d^{0,1} \circ \delta^{0,0})}{\ker(\delta^{0,0}) \cdot \ker(d^{0,0})}.
 \end{equation*}
Defining 
\begin{equation}\label{eq:TgM}
 T_\g(\mathcal{U}) :=\frac{\ker(d^{0,1} \circ \delta^{0,0})}{\ker(\delta^{0,0}) \cdot \ker(d^{0,0})}, \qquad T_\g(M) := \varinjlim T_\g(\mathcal{U}),
 \end{equation}
 gives us the relation between $\mathrm{Pic}_\g(M)$ and $\mathrm{Pic}^{\op{red}}_\g(M)$:

 \begin{prop} \label{prop:HtotPic}
The following sequence is exact:
 \begin{equation}\label{ll3}
0\to T_\g(M)\longrightarrow\mathrm{Pic}_\g(M)\xrightarrow{\varpi}\mathrm{Pic}^{\op{red}}_\g(M)\to 0.
 \end{equation}
 \end{prop}

The commutative diagram in Figure \ref{fig:cohom}  gives relations between the groups we have considered, %allowing to derive the vanishing conditions for $\op{Pic}_\g(M)$ and various isomorphisms.
leading to  vanishing conditions for $\op{Pic}_\g(M)$ and various isomorphisms.
 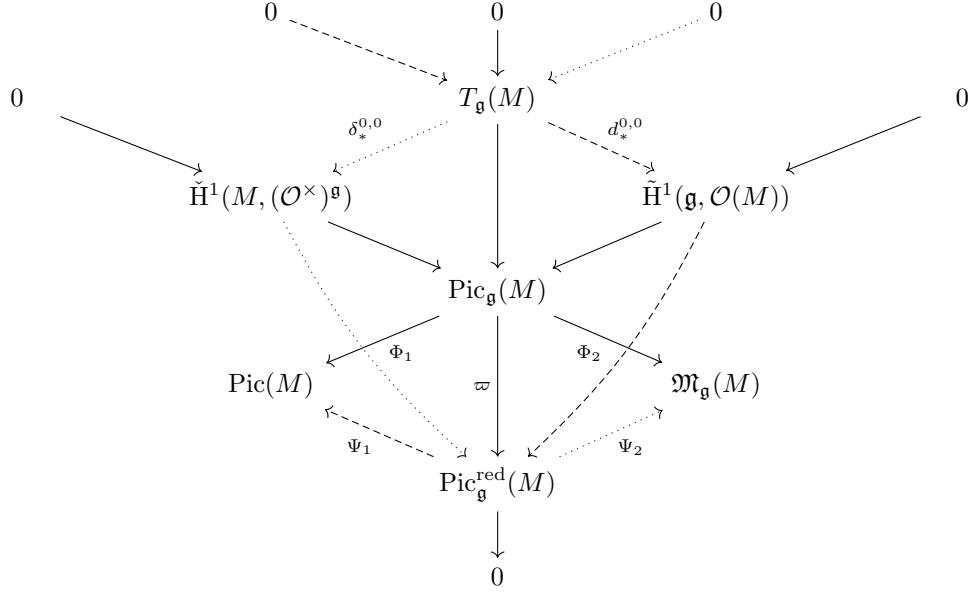
\begin{figure}[h]
{\small
\[\begin{tikzcd}
	& 0 & 0 & 0 \\
	 \quad \qquad 0 \quad \qquad   && {T_\g(M)} &&  \qquad \quad 0 \qquad \quad\\
	& {\check\HH^1(M, (\mathcal{O}^\times)^\g)} && {\Hmod^1(\g, \mathcal{O}(M))} \\
	&& {\mathrm{Pic}_\g(M)} \\
	& {\op{Pic}(M)} && {\mathfrak{M}_\g(M)} \\
	&& {\mathrm{Pic}^{\op{red}}_\g(M)} \\
	&& 0
	\arrow[from=2-3, to=4-3]
	\arrow["{\varpi}"', from=4-3, to=6-3]
	\arrow["{\Psi_1}", dashed, from=6-3, to=5-2]
	\arrow["{\Psi_2}"', dotted, from=6-3, to=5-4]
	\arrow["{\Phi_1}", from=4-3, to=5-2]
	\arrow["{\Phi_2}"', from=4-3, to=5-4]
	\arrow[from=1-3, to=2-3]
	\arrow[from=2-1, to=3-2]
	\arrow[from=2-5, to=3-4]
	\arrow[dashed, from=1-2, to=2-3]
	\arrow[dotted, from=1-4, to=2-3]
	\arrow[dotted, "{\delta_*^{0,0}}"', from=2-3, to=3-2]
	\arrow[dashed, "{d_*^{0,0}}", from=2-3, to=3-4]
	\arrow[from=3-2, to=4-3]
	\arrow[from=3-4, to=4-3]
	\arrow[from=6-3, to=7-3]
	\arrow[dashed, bend left=10, from=3-4, to=6-3]
	\arrow[dotted, bend right=10, from=3-2, to=6-3]
\end{tikzcd}\]}
\caption{Commutative diagram: The dotted and dashed long sequences as well as 
three straight line sequences are exact.} \label{fig:cohom}
 \end{figure}
The diagram contains the short exact sequence \eqref{ll3}, the (direct limit of) short exact sequences \eqref{ll1}-\eqref{ll2},
and also two longer exact sequences. For instance, exactness at $\delta^{0,0}_*$ and $d^{0,0}_*$ 
can be seen as follows.
If $\mu \in \ker(d^{0,1} \circ \delta^{0,0})$ then $\delta^{0,0} \mu \in \ker(d^{0,1}) \cap \ker(\delta^{0,1})$ 
and $d^{0,0} \mu \in \ker(\delta^{1,0}) \cap \ker(d^{1,0})$, whence 
$[\delta^{0,0} \mu] \in \HH_\delta^1(\HH_d^0(C^{\bullet,\bullet}))$ 
and $[d^{0,0} \mu] \in \HH_d^1(\HH_\delta^0(C^{\bullet,\bullet}))$. 
We have $[\delta^{0,0} \mu]=0$ if and only if there exists an element $\mu_0 \in \ker(d^{0,0})$ such that 
$\delta^{0,0} \mu = \delta^{0,0} \mu_0$. 
This happens if and only if $\mu=\frac{\mu}{\mu_0} \mu_0  \in \ker(\delta^{0,0}) \cdot \ker(d^{0,0})$. 
By a similar argument $[d^{0,0} \mu] = 0$ if and only if $\mu \in \ker(\delta^{0,0}) \cdot \ker(d^{0,0})$.

Note that the maps $\Hcech^1(M,(\mathcal{O}^\times)^\g) \to \mathrm{Pic}_\g(M)$ and $\Hmod^1(\g,\mathcal{O}(M)) \to \mathrm{Pic}_\g(M)$ in the commutative diagram are defined by $[g] \mapsto [(g,0)]$ and $[\lambda] \mapsto [(1,-\lambda)]$, respectively.

 \begin{cor}\label{cor:isomorphic}
If $\Hcech^1(M, (\mathcal{O}^\times)^\g)=0$ or $\Hmod^1(\g, \mathcal{O}(M))=0$, 
then $\mathrm{Pic}^{\op{red}}_\g(M)\simeq\mathrm{Pic}_\g(M)$.
 \end{cor}

\begin{cor}\label{cor:inj}
{\rm (i)} We have $\Hmod^1(\g, \mathcal{O}(M))=0$ if and only if $\Phi_1\colon \op{Pic}_\g(M)\to\op{Pic}(M)$ is injective.\\
{\rm (ii)} Likewise, $\check\HH^1(M,(\mathcal{O}^\times)^\g)=0$ if and only if
 $\Phi_2\colon \op{Pic}_\g(M)\to\mathfrak{M}_\g(M)$ is injective.
 \end{cor}

Notice $\mathrm{im}(\Psi_1)= \mathrm{im}(\Phi_1)$ and $\mathrm{im}(\Psi_2)= \mathrm{im}(\Phi_2)$,
which are described by Lemmata \ref{lem:proj} and \ref{lem:impsi}. 

% Now we interpret the kernel of the reduction map $\varpi$ as follows.

 \begin{prop}\label{prop:TgM}
The group $T_\g(M)$ of equivariant line bundles with trivial reduction
corresponds to global locally trivial lifts of $\g$ to the trivial line bundle over $M$
modulo globally trivial lifts.
 \end{prop}

 \begin{proof}
First note that an element $[(g,\lambda)]\in\op{Pic}_\g(M)$ in the kernel of $\varpi$ also belongs to the kernel 
of $\Phi_1=\Psi_1\circ\varpi$, so $g$ determines a trivial line bundle. Similarly using 
$\Phi_2=\Psi_2\circ\varpi$ we conclude that $\lambda$ yields a locally trivial lift (the multiplier
is cohomologous to zero on open sets $U_\alpha$). 

Next, applying $d^{0,0}$ to both the numerator and the denominator of the right hand side of \eqref{TgM}
we get $T_\g(M)\simeq\ker\bigl(\delta^{1,0}|_{\op{im}d^{0,0}}\bigr)/d^{0,0}(\ker\delta^{0,0})$,
whence the required interpretation.

Finally, by applying $\delta^{0,0}$ to \eqref{TgM} we conclude 
$T_\g(M)\simeq\ker\bigl(d^{0,1}|_{\op{im}\delta^{0,0}}\bigr)/\delta^{0,0}(\ker d^{0,0})$,
which corresponds to line bundles with $\g$-invariant transition functions modulo global $\g$-invariants.
 \end{proof}
Propositions \ref{prop:HtotPic} and \ref{prop:TgM} combine into Theorem \ref{Th1}.

 \begin{example}[Rational curve]\label{ex:sl2P1}
Consider the projective space $\mathbb CP^1$ with charts $U_0\simeq\mathbb C^1(x)$ and 
$U_\infty\simeq\mathbb C^1(y)$, with coordinates related by $y=1/x$ on $U_0 \cap U_\infty$. 
The Lie algebra $\mathfrak{sl}(2,\mathbb C)$ acts naturally on this space
with the basis $X,Y,Z$ given in local coordinates:
 \[ 
X|_{U_0} = \partial_x, \;\; Y|_{U_0} = x \partial_x, \;\; Z|_{U_0} = x^2 \partial_x, \qquad\quad 
X|_{U_\infty} = -y^2 \partial_y, \;\; Y|_{U_\infty} = -y\partial_y, \;\; Z|_{U_\infty} = -\partial_y.  
 \]
Let $\lambda_i$ be a representative of an element in $H^1(\mathfrak{sl}(2,\mathbb C),\mathcal{O}(U_i))$ for 
$i=0,\infty$.
	
Taking $\mu_0= e^{\int \lambda_0(X_0) dx}$ (the integral sign denotes the anti-derivative on $\mathbb C$) gives
 \[ 
(\lambda_0-d^{0,0} \mu_0)(X_0) =\lambda_0(X_0)-\partial_x (\log \mu_0) =0, 
 \]
so we can without loss of generality assume that $\lambda_0(X)=0$. 
The values $\lambda_0(Y)$ and $\lambda_0(Z)$ are now determined by the cocycle conditions
	\begin{align*}
		X(\lambda_0(Y))-Y(\lambda_0(X))&=\lambda_0 ([X,Y]) = \lambda_0(X) = 0, \\
		X(\lambda_0(Z))-Z(\lambda_0(X))&=\lambda_0 ([X,Z]) = 2 \lambda_0(Y), \\
		Y(\lambda_0(Z))-Z(\lambda_0(Y))&=\lambda_0 ([Y,Z]) =  \lambda_0(Z).
	\end{align*}
This leads to
 \[
\lambda_0(X)=0, \qquad \lambda_0(Y)=\tfrac12 A, \qquad \lambda_0(Z) = A x.
 \]
Analogous computations on $U_\infty$ gives
 \[
\lambda_\infty(X)=By, \qquad \lambda_\infty(Y)=\tfrac12 B, \qquad \lambda_\infty(Z) = 0.
 \]
% Note that the cocycles $\lambda_0$ and $\lambda_\infty$ after the restriction to $U_0\cap U_\infty$
% are cohomologous (differing by $d\log x^{2A}$) iff $A=-B$. 
 
Next we require that $d^{0,1} g_{0 \infty} = (\delta^{1,0} \lambda)_{0 \infty}$ for some $\delta^{0,1}$-cocycle 
$g_{0 \infty} \in \mathcal{O}^\times(U_0 \cap U_\infty)$. Evaluating this on $X$, $Y$ and $Z$ leads to the following 
overdetermined system of ODEs:
 \[ 
\frac{\partial_x (g_{0 \infty})}{g_{0 \infty}} = -\frac{B}{x}, \qquad 
\frac{x\partial_x (g_{0 \infty})}{g_{0 \infty}} = \frac{A-B}2 , \qquad 
\frac{x^2\partial_x (g_{0 \infty})}{g_{0 \infty}} = Ax . 
 \]
The system has a solution if and only if $A=-B$ in which case $g_{0 \infty} = C x^A$. 
This solution is holomorphic on $U_0 \cap U_\infty$ if and only if $A \in \mathbb Z$. 
The constant $C$ can be set equal to $1$ by multiplying $g_{0 \infty}$ with  $(\delta^{0,0} \mu)_{0 \infty}$ for 
$\mu = \{\mu_0=1, \mu_\infty=C\} \in \ker(d^{0,0})$. Thus the global lifts are given by 
$(\lambda_0, \lambda_\infty)$ with $A=-B \in \mathbb Z$, and the corresponding 
$\mathfrak{sl}(2,\mathbb C)$-equivariant line bundle has transition function $g_{0 \infty} = x^{A}$. 
To sum up, the cover $\mathcal{U}=\{U_0,U_\infty\}$ is nice and we get
 \[
\mathrm{Pic}_{\mathfrak{sl}(2,\mathbb C)}(\mathbb CP^1)\simeq 
\mathrm{Pic}^{\op{red}}_{\mathfrak{sl}(2,\mathbb C)}(\mathbb CP^1)\simeq \mathbb Z.
 \]
The first isomorphism  is a consequence of $\Hcech^1(M, (\mathcal{O}^\times)^\g) \simeq \HH_\delta^1(\HH_d^0(C^{\bullet,\bullet}))  =0$ and 
Corollary \ref{cor:isomorphic}. 
Since $\mathbb CP^1$ is covered by two open charts, we have $C^{0,2}=0$, implying $\HH_d^0(C^{\bullet,2})=0$ 
and $\HH_\delta^2(\HH_d^0(C^{\bullet,\bullet}))=0$. By Lemma \ref{lem:impsi}, $\mathrm{im}(\Psi_2) \simeq \HH_\delta^0(\HH_d^1(C^{\bullet,\bullet})) = \mathfrak{M}_\g(\{U_0,U_\infty\})$, which  for this cover is isomorphic to $\mathbb Z$.

A straightforward generalization of this computation gives $T_\g(\C P^n)=0$, 
$\mathfrak{M}_\g(\C P^n)=\C$ (for $n=1$ this is parametrized by the above $A=-B$, but with a finer cover 
$\mathcal{U}$ it is unconstrained: $A\in\C$)
and $\op{Pic}_\g(\C P^n)=\mathbb Z$ for $\g=\mathfrak{sl}(n+1,\C)$. 
 \end{example}

 \begin{remark} \label{rk:tautological}
Recall that $\mathrm{Pic}(\C P^n) = \{\mathcal{O}_{\C P^n}(k)\}_{k\in\mathbb Z}\simeq\mathbb Z$, where 
$\mathcal{O}_{\mathbb CP^n}(0)$ is the trivial line bundle,
$\mathcal{O}_{\mathbb CP^n}(-1)$ is the tautological line bundle and for $k>0$:
 \[
\mathcal{O}_{\mathbb CP^n}(-k) = \mathcal{O}_{\C P^n}(-1)^{\otimes k}, \qquad 
\mathcal{O}_{\mathbb CP^n}(k) =\bigl(\mathcal{O}_{\C P^n}(-1)^{\otimes k}\bigr)^*.
 \] 
The canonical line bundle is $K_{\mathbb CP^n}=\Lambda^n T^* \mathbb CP^n=
\mathcal{O}_{\mathbb CP^n}(-n-1)$, cf.\ \cite[Ch.\ 2.2]{H}.
 \end{remark}

We will see later that $\Psi_1$ and $\Phi_1$ may be non-injective. The following shows it for $\Psi_2$
and $\Phi_2$.

 \begin{example}[Elliptic curve]
Consider $\mathbb C^2$ with coordinates $(x,u)$ and two commuting maps
 \[
h_1(x,u)= (x+1,u), \qquad h_2(x,u)= (x+\omega_1, \omega_2 u),
 \]
where $\omega_1\in\C\setminus\R$, $\omega_2\neq0$. 
Both of these maps respect the projection $\mathbb C^2 \to \mathbb C$ given by 
$(x,u)\mapsto x$, and the vector field $\partial_x$. Thus in the quotient by the $\mathbb Z^2$ action 
generated by $h_1,h_2$ we get that the vector field $\partial_x$ on the elliptic curve $\Gamma = \mathbb C/\mathbb Z^2$ lifts to the vector field $\partial_x$ on the line bundle $\mathbb C^2/\mathbb Z^2$ over the elliptic curve. 
This line bundle $L_\omega$ is topologically trivial but holomorphically nontrivial for 
$\omega_2 \neq e^{2 \pi i \omega_1}$, and all line bundles of this form lie in $\ker(\Psi_2)$; 
see section 27 of \cite{A} for details. The general lift of $\g$ is given by $\p_x+c\,u\p_u$, $c\in\C$.

For holomorphic curves we have a short exact sequence (where $c_1$ is the first Chern class)
 \[
0\to\op{Pic}^0(\Gamma)\longrightarrow\op{Pic}(\Gamma)\stackrel{c_1}\longrightarrow H^2(\Gamma,\mathbb Z)\to0,
 \]
and for elliptic curves $\op{Pic}^0(\Gamma)=\op{Div}^0(\Gamma)\simeq\Gamma$, $x\mapsto x-x_0$,
whence $\op{Pic}(\Gamma)\simeq\Gamma\oplus\mathbb Z$. 

The summand $\mathbb Z$ in $\op{Pic}(\Gamma)$ corresponds to divisors $m\cdot x_0$, $m\in\mathbb Z$, $x_0\in\Gamma$. However for topologically nontrivial line bundles, $m=c_1(L)\neq0$, the algebra $\g$
does not possess a lift to $L$. Indeed, such a lift would define a flat connection, at which point we
can use the formula $c_1(L)=\Bigl[\frac{-1}{2\pi i}\op{tr}R_\nabla\Bigr]$.
Alternatively, denoting by $\pi:\C\to\Gamma$ the quotient-projection by the lattice $\langle1,\omega_1\rangle$,
the pullback $\pi^*L$ is trivial and can be identified with $\C^2(x,u)$, on which $\mathbb Z^2$ acts through 
the above $h_1,h_2$. Invariance of the lift $\p_x+f(x)\,u\p_u$ gives periodicity $f(x+1)=f(x)$ and the constraint
$f(x+\omega_1)-f(x)=2\pi im$, which are incompatible unless $m=0$.

It is easy to see that $\mathfrak{M}_\g(\Gamma)=0$. Moreover 
$T_\g(\Gamma)=\C$ as it corresponds to 0-cochains $c_\alpha e^{sx}\in\mathcal{O}(U_\alpha)$
modulo local constants $\{c_\alpha\}$ (so the quotient coordinate is $s$). This can be also identified with 
 $\Hmod^1(\g,\mathcal{O}(M))=\C$ generated by global 1-form $dx$ on $\Gamma$.

We conclude: 
 \[
\op{Pic}^{\op{red}}_\g(\Gamma)=\Gamma,\quad
\op{Pic}_\g(\Gamma)=\C^2/\mathbb Z^2.
 \]
Note that the equivariant Picard group can be identified with 
 $\Hcech^1(M,(\mathcal{O}^\times)^\g) =(\C^\times)^2$ but simultaneously 
it corresponds to trivial one-dimensional bundle, with fibers $\C(c)$, over $\Gamma$.
This fits well the commutative diagram of Figure \ref{fig:cohom}.
 \end{example}
 
Corollary \ref{cor:inj} gives a sufficient condition for $\Phi_1$ to be injective. For a connected algebraic group $G$,   Mumford's
Proposition 1.4 in \cite{M} gives a sufficient condition for the map $\mathrm{Pic}_G(M)\to\mathrm{Pic(M)}$ 
to be injective, in terms of non-existence of a homomorphism $G \to GL(1,\mathbb C)$. 
(The group of $G$-equivariant line bundles will be discussed in Section \ref{sect:Gequivar}.)
This does not straight-forwardly adapt to the infinitesimal analytic setting, 
yet below we obtain a result inspired by that of Mumford. 

Let $\g_p\subset\g$ denote the isotropy algebra of the point $p\in M$: 
 \[ 
\g_p = \{X \in \g \mid X_p = 0\}.
 \]
Let $\HH^k_{dR}(M)$ denote the holomorphic de Rham cohomology of $M$.
It is known that in the affine case (for Stein manifolds) as well as for the compact K\"ahler case
this coincides with the singular cohomology $\HH^k(M,\C)$, see \cite{Gr,GH}. % ... Hartshorne, Voisin ...
In general, the holomorphic de Rham cohomology $\HH^k_{dR}(M)$ is equal to 
the hypercohomology $\mathbb{H}^k(M,\Omega^\bullet_M)$ of the sheaf of holomorphic forms on $M$.

 \begin{lemma} \label{lem:stab}
Let $\g$ be a transitive Lie algebra of vector fields on a manifold $M$. Define
$Z=\{[\lambda]\in \Hmod^1(\g, \mathcal{O}(M)) \mid \lambda(Y)_p=0\,\forall Y \in \g_p,\forall p\in M\}$
(the defining property is representative-independent). 
Then we have a natural embedding $Z\hookrightarrow \HH^1_{dR}(M)$.

In particular, if $\HH^1_{dR}(M)=0$ then every $[\lambda]\in Z$ is exact:  
$\lambda=d^{0,0}\mu$ for some $\mu\in \ker(\delta^{0,0})$.
 \end{lemma}

 \begin{proof}
 Consider $[\lambda] \in Z$. Since its representative $\lambda$ is defined globally on $M$, the value $\lambda(X)_p \in \mathbb C$ is well-defined for any $p \in M$ 
and any $X \in \g$. Transitivity of $\g$ implies that 
for every $v \in T_p M$ there exists $X\in\g$ satisfying $X_p=v$, which means that
the following linear function on $T_p M$ is well-defined:
 \[ 
\alpha_p \colon X_p \mapsto \lambda(X)_p.
 \]
This gives a well-defined 1-form $\alpha$ on $M$.
Choosing $X_i\in\g$ for $v_i\in T_pM$ such that $v_i=(X_i)_p$ we get for every $p\in M$:
 \begin{align*}
(d \alpha)_p(v_1, v_2) &= d \alpha (X_1,X_2)_p = 
\left( X_1(\alpha(X_2)) -X_2(\alpha(X_1)) - \alpha([X_1,X_2])\right)_p
= d^{1,0} \lambda(X_1,X_2)_p = 0.
 \end{align*}
Thus $d\alpha=0$, so the closed 1-forms in two cohomologies correspond; the same clearly concerns exact 1-forms.
This yields the embedding $[\lambda]\mapsto[\alpha]$.

If $\HH^1_{dR}(M)=0$ then $\alpha=d\log\mu$ for some $\mu\in\mathcal{O}^\times(M)$, 
and therefore $\lambda=d^{0,0}\mu$.
 \end{proof}

 \begin{prop}\label{prop:stab}
Let $\g$ be a transitive Lie algebra of vector fields on a manifold $M$ such that
$ \dim \Hmod^1(\g,\mathcal{O}(M))>\dim\HH^1_{dR}(M)$.
This holds, for instance, when $  \Hmod^1(\g,\mathcal{O}(M))\neq 0$ and $\HH^1_{dR}(M)=0$.
 % Then for a generic $p\in M$ 
Then for all $p\in M$ there exists a surjective Lie algebra homomorphism $\g_p\to\mathfrak{gl}(1,\mathbb C)$.
 \end{prop}

 \begin{proof}
Consider an element $[\lambda]\in\Hmod^1(\g,\mathcal{O}(M)) \setminus Z$ 
(by our assumption and Lemma \ref{lem:stab} this set is nonempty). 
Since $\delta^{1,0} \lambda=0$, $\lambda$ defines a global lift of $\g$ to the trivial bundle 
$M \times \mathbb C$:
\[ \g^\lambda = \{ X + \lambda(X) u \partial_u \mid X \in \g\}.\]
For any point $p \in M$, all vector fields of the Lie subalgebra
 \[
\{Y + \lambda(Y) u \partial_u \mid Y \in \g_p\} \subset \g^\lambda 
 \]
are tangent to the fiber $\pi^{-1}(p)$. Therefore the restriction to the fiber results in a Lie algebra homomorphism:
 \begin{equation}\label{Y456}
\g_p\ni Y \mapsto \lambda(Y)_p u \partial_u \in \mathfrak{gl}(1,\C).
 \end{equation}
By definition of $Z$, $\lambda(Y)_p\not\equiv0$ for a generic point $p$.   
Thus, at this point \eqref{Y456} is surjective. 
Since, for a transitive $\g$, the isotropies $\g_p$ at different points $p$ are conjugate, the claim follows.
 \end{proof}

The existence of a surjective Lie algebra homomorphism to $\mathfrak{gl}(1,\mathbb C)$ implies that 
$\g_p$ has an ideal of codimension 1, the kernel of \eqref{Y456}. 
 % Thus if $\g_p$ does not have an ideal of codimension 1 for any $p \in M$, 
 % it follows that $\HH^1_d(\HH_\delta^0(C^{\bullet,\bullet})) =0$.  
Corollary \ref{cor:inj} then leads to the following statement.

 \begin{cor}
Let $\g$ be a transitive Lie algebra of vector fields on a manifold $M$ with $\HH^1_{dR}(M)=0$. 
If $\g_p$ does not have an ideal of codimension 1, then $\mathrm{Pic}_\g(M) \to \mathrm{Pic}(M)$ is injective.
 \end{cor}
As a consequence, $\Phi_1$ is injective if $\g_p$ is perfect $[\g_p,\g_p]=\g_p$ 
(this also applies to infinite-dimensional Lie algebras $\g$) and,
in particular, if $\g_p$ is semisimple (for finite-dimensional $\g$).

 \begin{example}[Special affine algebra on the plane]
Consider the Lie algebra
 \[ 
\g = \langle \partial_x, \partial_y, y \partial_x, x\partial_y, x \partial_x-y\partial_y\rangle \subset \vf(\mathbb C^2).
 \]
The Lie algebra is transitive with simple isotropy $\g_p$. By Proposition \ref{prop:stab} we have 
$\Hmod^1(\g,\mathcal{O}(M))=0$, implying that $\mathrm{Pic}_\g(\mathbb C^2)\to\mathrm{Pic}(\mathbb C^2)$ 
is injective by Corollary \ref{cor:inj}. Since $\mathrm{Pic}(\mathbb C^2) =0$ it follows that $\mathrm{Pic}_\g(\mathbb C^2) = 0$. 

Next, for the Lie algebra
 \[ 
\mathfrak{h}=\langle y \partial_x, -xy \partial_x-y^2 \partial_y, \partial_x,-x^2 \partial_x-xy\partial_y,2x\partial_x+y\partial_y   \rangle \subset \vf(\mathbb C^2)
 \]
the isotropy of the point $p=0$ 
 \[
\mathfrak{h}_0 = \langle y \partial_x, -xy \partial_x-y^2 \partial_y,-x^2 \partial_x-xy\partial_y,2x\partial_x+y\partial_y \rangle
 \]
is 4-dimensional and solvable, and it has a 3-dimensional ideal. 
In this case $\mathrm{Pic}_{\mathfrak{h}}(\C^2)=\C$.  

Note that $\g$ and $\mathfrak{h}$ can be viewed as the same Lie subalgebra 
$\mathfrak{sl}(2,\C)\ltimes\C^2\subset\mathfrak{sl}(3,\C)\subset\vf(\mathbb CP^2)$
restricted to two different open charts of $\mathbb CP^2$.
 \end{example}

\subsection{Line bundles admitting a transversal lift}\label{sect:invariantdivisors}

We start by recalling some basic information about divisors, cf.\ \cite{GH,H}. Let $\mathcal{O}^\times$ denote 
the multiplicative sheaf of nonvanishing holomorphic functions on a complex manifold $M$, and 
$\mathcal{M}^\times$ the sheaf of meromorphic functions that are not identically zero on $M$.  
A divisor $D$ is a global section of $\mathcal{M}^\times/\mathcal{O}^\times$. It is defined by a collection 
of functions $f_\alpha \in \mathcal{M}^\times(U_\alpha)$ for an open cover $\{U_\alpha\}$ of $M$, such that 
$f_\alpha/f_\beta \in \mathcal{O}^\times(U_\alpha \cap U_\beta)$.  
Any divisor $D$ gives rise to a line bundle, denoted by $[D]$, whose transition functions are given by 
$g_{\alpha \beta} = f_\alpha/f_\beta \in \mathcal{O}^\times(U_\alpha \cap U_\beta)$, and 
the long exact sequence (see \cite{GH}) relates the group of divisors
$\mathrm{Div}(M) := \Hcech^0 (M,\mathcal{M}^\times/\mathcal{O}^\times)$ to the Picard group on $M$:
 \begin{equation}\label{eq:longexactpic}
\cdots\to\Hcech^0(M,\mathcal{M}^\times)\to\op{Div}(M)\to\op{Pic}(M)\to\Hcech^1(M,\mathcal{M}^\times)\to\cdots
  \end{equation}
Here $\Hcech^0(M, \mathcal{M}^\times)$ is the group of  global meromorphic functions on $M$, and
$\op{Div}(M)/\Hcech^0(M,\mathcal{M}^\times)$ is the group of equivalence classes of divisors
(equivalent divisors give equivalent line bundles).

 \begin{definition}\label{def:relinv}
Let $\mathfrak g \subset \vf(M)$ be a Lie algebra of vector fields on $M$. The divisor $D=\{f_\alpha\}$ 
defined on the open cover $\{U_\alpha\}$ of $M$ is a $\g$-invariant divisor if for each $\alpha$
 \[
X(f_\alpha)=\lambda_\alpha(X) f_\alpha, \qquad \forall X \in \g,
 \]
where $\lambda_\alpha \in \g^* \otimes \mathcal O(U_\alpha)$. The group of $\g$-invariant divisors is denoted by 
$\mathrm{Div}_\g(M)$. The collection $\lambda=\{\lambda_\alpha\}$ is called the weight of $D$.
 \end{definition}
It follows that $\mathfrak g$ is tangent to the set of zeros of $D$, and also to the set of poles. In this way $D$ defines a (possibly reducible) invariant hypersurface in $M$.

 %  \begin{remark}
 % Notice the similarity between the compatibility condition \eqref{eq:compatibility} and the equation in Definition \ref{def:relinv}
 % \end{remark}

 \begin{prop} \label{prop:tobundle}
Let $\g \subset \vf(M)$ be a Lie algebra of vector fields on $M$, and let  $D=\{f_\alpha\}$ be a $\g$-invariant 
divisor with weight $\lambda=\{\lambda_\alpha\}$. Set $g_{\alpha \beta} = f_\alpha/f_\beta$ and define 
$g=\{g_{\alpha \beta}\}$. Then the pair $(g, \lambda)$ defines a $\g$-equivariant line bundle $L=[D]$,  which is independent of the choice of representative functions $f_\alpha$.
 \end{prop}

 \begin{proof} 
To show that the pair defines a $\g$-equivariant line bundle, we must verify that $\partial^1(g,\lambda)=0$. It is clear that $\delta^{0,1} g=0$ since $g_{\alpha \beta}$ 
are transition functions of $[D]$. Next, the condition $d^{1,0} \lambda_\alpha =0$ holds for each $\alpha$ since 
$f_\alpha\not\equiv0$ and for arbitrary vector fields $X,Y\in\g$ we have
 \[
\lambda_\alpha([X,Y]) f_\alpha = [X,Y](f_\alpha) = X(Y(f_\alpha))-Y(X(f_\alpha)) = (X(\lambda_\alpha(Y))-Y(\lambda_\alpha(X)))f_\alpha.
 \]
What remains is to verify that the weights $\lambda_\alpha$ are compatible with the transition functions 
$g_{\alpha \beta}= f_\alpha/f_\beta$. On $U_\alpha \cap U_\beta$ we have
 \[ 
\lambda_\alpha(X) f_\alpha = X(f_\alpha) = X(g_{\alpha \beta} f_\beta) = X(g_{\alpha \beta}) f_\beta + g_{\alpha \beta} X(f_\beta) = \left(\frac{X(g_{\alpha \beta})}{g_{\alpha \beta}}+\lambda_\beta(X)\right) f_\alpha,
 \]
which is equivalent to $\delta^{1,0}\lambda = d^{0,1} g$.

Next, to show that the $\g$-equivariant bundle is independent of representative functions $f_\alpha$ of $D$, take another representative  $\tilde f_\alpha = \mu_\alpha f_\alpha$ with $\mu_\alpha \in \mathcal{O}^\times(U_\alpha)$. This results in an equivalent $\g$-equivariant line bundle $(\{\tilde g_{\alpha \beta}\}, \{\tilde \lambda_\alpha\})$:  $\tilde g_{\alpha \beta} = g_{\alpha \beta} \mu_\alpha/\mu_\beta$ and $\tilde \lambda_\alpha(X) = \lambda_\alpha(X) + X(\mu_\alpha)/\mu_\alpha$ for all $X \in \g$.
 \end{proof}

%The divisors defined respectively by $f_\alpha$ and $\tilde f_\alpha$ are equiv if and only if $\tilde f_\alpha=e^{\mu_\alpha} f_\alpha$.  Since
%\[X(\tilde f_\alpha)=(\lambda_\alpha(X)+d\mu_\alpha(X)) \tilde f_\alpha,\]
%it follows that two equivalent divisors defined on the same charts are equivalent if and only if their cocycles lie in the same cohomology class. Thus $\lambda_\alpha \in \HH^1(\mathfrak g,\mathcal{O}(U_\alpha))$.

 % The line bundle of Proposition \ref{prop:tobundle} is $\pi\colon [D] \to M$ with transition functions 
 % $g_{\alpha \beta} = f_\alpha/f_\beta$, while 
The lifted Lie algebra $\g^{\lambda} \subset \vf([D])$ is given locally on 
$\pi^{-1}(U_\alpha) \simeq U_\alpha \times \mathbb C$ by
 \[
\g^{\lambda}|_{\pi^{-1}(U_\alpha)} =  \{\hat X|_{\pi^{-1}(U_\alpha)}=X|_{U_\alpha}+\lambda_\alpha(X) u \partial_u \mid X \in \g\}.
 \]
The lift $\g^\lambda$ has exactly the same form as the lifts in Section \ref{sect:2.1}, the only difference 
being that $\lambda$ is now specifically determined by a $\g$-invariant divisor. Thus we have a map
 \[ 
j_\g \colon \mathrm{Div}_\g(M)\to\mathrm{Pic}_\g(M),\qquad
D\mapsto([D], \g^\lambda).
 \]
In general this map is neither injective nor surjective. For instance, the kernel of this map contains all global 
$\g$-invariant functions in $\mathcal{M}^\times(M)$. 
If $\g$ is transitive, then $\mathrm{Div}_\g(M)=0$, but $\mathrm{Pic}_\g(M)$ may be nontrivial
as in Example \ref{ex:sl2P1}. The next example exhibits non-trivial $\mathrm{Div}_\g(M)$.

 \begin{example}[$\mathfrak{sl}(2,\mathbb C) \subset \vf(\mathbb CP^2)$]
The manifold $\mathbb CP^2$ is covered by the three charts
 \begin{align*}
U_3 &= \{[x:y:z] \in \mathbb CP^2 \mid z \neq 0\}, \\ 
U_2 &= \{[x:y:z] \in \mathbb CP^2 \mid y \neq 0\} , \\ 
U_1 &= \{[x:y:z] \in \mathbb CP^2 \mid x \neq 0\},
 \end{align*}
on which coordinates are given respectively by
 \[
(x_1,x_2)=(x/z,y/z), \qquad (y_1,y_3) = (x/y,z/y), \qquad (z_2,z_3) = (y/x,z/x).
 \]
Consider the Lie algebra $\mathfrak{sl}(2,\mathbb C) \subset \mathfrak{sl}(3,\mathbb C)$ given in the respective charts by
 \begin{gather*}
\langle x_2\partial_{x_1}, x_1\partial_{x_2}, x_1\partial_{x_1}-x_2\partial_{x_2} \rangle, \\
\langle\partial_{y_1}, -y_1^2\partial_{y_1}-y_1 y_3\partial_{y_3}, 2y_1\partial_{y_1}+y_3\partial_{y_3} \rangle, \\
\langle -z_2^2\partial_{z_2}-z_2 z_3\partial_{z_3},\partial_{z_2}, -2z_2\partial_{z_2}-z_3\partial_{z_3} \rangle.
 \end{gather*}
A computation shows that the Chevalley-Eilenberg cohomology groups are  
 \[
\HH^1 (\mathfrak{sl}(2,\C),\mathcal{O}(U_3)) = 0, \qquad  
\HH^1 (\mathfrak{sl}(2,\C),\mathcal{O}(U_2)) = \C^2, \qquad 
\HH^1 (\mathfrak{sl}(2,\C),\mathcal{O}(U_1)) = \C^2
 \]
with representative cocycles
 \[ 
\lambda_3 = (0,0,0), \qquad 
\lambda_2= (0, B_1 y_1+B_2 y_3^2,-B_1), \qquad 
\lambda_1 = (C_1 z_2+C_2 z_3^2,0,C_1),
 \]
The holomorphic transition functions, compatible via overdetermined system \eqref{eq:compatibility}, 
exist only for $B_2=C_2=0$, $C_1=B_1=b$, and are given by formulae
 \[ 
g_{32} = A_1 y_3^{b}=A_1 x_2^{-b}, \qquad 
g_{31} = A_2 z_3^{b}=A_2 x_1^{-b}, \qquad 
g_{21} = A_3 z_2^{b} = A_3 y_1^{-b}.
 \]
Requiring $g_{\alpha \beta}$ to be holomorphic gives the further restriction $b\in\mathbb Z$. The constants $A_1, A_2, A_3$ can be set equal to $1$ 
by multiplying with an $\mathfrak{sl}(2,\mathbb C)$-invariant $\delta^{0,0}$-coboundary. We conclude:
 \[ 
\mathrm{Pic}_{\mathfrak{sl}(2,\C)}(\mathbb CP^2)\simeq\HH^1(\mathrm{Tot}^\bullet(C))=\mathbb Z.
 \]
In this case $\mathrm{Div}_{\mathfrak{sl}(2,\C)}(\C P^2)\simeq\mathrm{Pic}_{\mathfrak{sl}(2,\C)} (\C P^2)$. 
The unique divisor $D=\{f_1,f_2,f_3\}$ corresponding to $b\in\mathbb Z$ is given by
 \[
f_1 = z_3^{-b}, \qquad f_2 = y_3^{-b}, \qquad f_3=1.
 \]
 \end{example}

We will now describe an obstruction for the existence of invariant divisors, elaborating upon \cite{FO}. 
The following definition is adapted from \cite{AF} where it was used for group actions.

 \begin{definition}
For a Lie algebra $\g\subset\vf(M)$ and a holomorphic line bundle $\pi \colon L \to M$,
a lift $\hat\g\subset\vf(L)$ is called transversal if generic $\hat \g$-orbits on $L$ 
$\pi$-project biholomorphically.
 \end{definition}

Note that singular orbits of $\hat\g$ may project non-injectively (but indeed surjectively) to $\g$-orbits 
on $M$ (see e.g.\ Example \ref{ex:aff1} below). The following is a reformulation of Theorem~\ref{Th2}.

\begin{prop} \label{prop:orbitdim}
Let $(L,\hat \g)$ be a $\g$-equivariant line bundle over $M$. Suppose $L\in\op{im}(\Psi_1\circ j_\g)$, i.e.,
$L=[D]$ for some $\g$-invariant divisor $D$ with weight $\lambda$ and $\hat \g = \g^\lambda$. 
Then $\hat \g$ is transversal.
 \end{prop}

 \begin{proof}
Let $D=\{f_\alpha\}$ be a $\mathfrak g$-invariant divisor with weight $\lambda=\{\lambda_\alpha\}$, 
and $[D]$ the corresponding line bundle defined by transition functions $g_{\alpha \beta} = f_\alpha/f_\beta$. 
Any element $\hat X \in \g^\lambda$ takes on $U_\alpha$ the form $\hat X|_{U_\alpha} = X|_{U_\alpha} + \lambda_\alpha(X) u \partial_u$ for some $X \in \g$. A straight-forward computation shows that
 \[ 
\hat X(u/f_\alpha) = \frac{\hat{X}(u) f_\alpha - u \hat{X}(f_\alpha)}{f_\alpha^2} = \frac{u \lambda_\alpha(X) f_\alpha-u X(f_\alpha)}{f_\alpha^2} =0.
 \]
Thus, the function $u/f_\alpha$ on $U_\alpha \times \mathbb C$ is a meromorphic absolute invariant 
(constant on $\g^\lambda$-orbits). It follows that the dimension of generic $\mathfrak g$-orbits on $U_\alpha$ 
is equal to the dimension of generic $\g^\lambda$-orbits on $U_\alpha \times \mathbb C$. 
This holds simultaneously on each $U_\alpha$, and therefore globally on $M$.
 \end{proof}

 \begin{remark}
Local absolute invariants $u/f_\alpha$ define a collection of local sections tangent to $\g^\lambda$, 
which are given by $u=C f_\alpha$ with $C$ being an absolute $\g$-invariant. 
Choosing $C$ to be a global invariant on $M$ gives a global $\g^\lambda$-invariant section of $[D]$.
 \end{remark}

Returning to Example \ref{ex:sl2P1} on $\mathfrak{sl}(2,\C)\subset\vf(\C P^1)$, we observe
that generic orbits of any nontrivial lift are 2-dimensional. Thus there are no nontrivial invariant divisors,
which also follows from the fact that $\mathfrak{sl}(2,\mathbb C)$ is transitive on $\mathbb CP^1$. 
Here is another demonstration of Proposition \ref{prop:orbitdim}.

 \begin{example}[$\mathfrak{aff}(1,\mathbb C) \subset \vf(\mathbb CP^1)$] \label{ex:aff1}
Consider again coordinate charts $U_0 \simeq \mathbb C^1(x)$ and $U_\infty \simeq \mathbb C^1(y)$
of $\mathbb CP^1$, with the Lie subalgebra $\g=\mathfrak{aff}(1,\C)=\langle X,Y\rangle\subset\mathfrak{sl}(2,\C)$ given by
 \[ 
X|_{U_0}=\partial_x,\ Y|_{U_0}=x\partial_x, \qquad X|_{U_\infty}=-y^2\partial_y,\ Y|_{U_\infty}=-y\partial_y.
 \]
General representatives $\lambda_s$ of elements in $\HH^1(\mathfrak{aff}(1,\C),\mathcal{O}(U_s))$,
for $s=0,\infty$, in basis $(X,Y)$ are given by
 \[ 
\lambda_0=(0,A), \qquad  \lambda_\infty=(B_2 y,B_1),\qquad A,B_1,B_2\in\C.
 \]
%More generally, for $\g=\mathfrak{aff}(1,\C)$ the sheaf $\mathfrak{M}_\g$ consists of 
%maps $U\mapsto\lambda_U\in\C^{1+d_U}$, where $d_U=1$ for $\infty\in U$ and $d_U=0$ else.
%We however continue with the cover $\C P^1=U_0\cup U_\infty$.
A general compatible transition function exists only when $A=B_1-B_2$, in which case it is 
 cohomologous to $g_{0 \infty}=y^{B_2}=x^{-B_2}$. Requiring $g_{0 \infty}$ to be holomorphic 
results in $B_2\in\mathbb Z$. The local lifts corresponding to $\lambda_0$ and $\lambda_\infty$ are given by
 \[  
\g^{\lambda_0} = \langle \partial_x, x\partial_x + (B_1-B_2) u \partial_u \rangle, \qquad  
\g^{\lambda_\infty} = \langle -y^2 \partial_y + B_2 y u \partial_u, -y \partial_y + B_1 u \partial _u\rangle.
 \]
It is clear that the generic orbit dimension is 1 if and only if $B_2=B_1$. 
In this case we get the invariant divisor $D$ given by $f_0=1$ and $f_\infty=y^{-B_1}$.
Thus 
 \[
\op{Div}_\g(\C P^1)=\mathbb{Z}\subsetneqq\C\times\mathbb{Z}=\op{Pic}_\g(\C P^1).
 \]
Note that the map $\Psi_1\colon \mathrm{Pic}_\g(\C P^1)\to\mathrm{Pic}(\C P^1)$ is not injective:
$\ker(\Psi_1) \simeq \HH_d^1(\HH_\delta^0(C^{\bullet,\bullet}))=\C$.  Similar to Example \ref{ex:sl2P1}, we have $\mathfrak{M}_\g(M) = \mathbb C$ even though $\mathfrak{M}_\g(\{U_0, U_\infty\}) = \mathbb Z$. 
 \end{example}

Proposition \ref{prop:orbitdim} can be viewed as a global version of \cite[Th.\ 5.4]{FO}. 
According to it, locally, in smooth regular case the statement allows a converse,
giving a criterion for the (local) existence of relative invariants.
Globally, in general analytic context, there is no converse to Proposition \ref{prop:orbitdim}, due to
other reasons for non-existence of  meromorphic invariant divisors/relative invariants.  
This is shown in the following simple example, and also in a more complicated example of Section \ref{sect:projectivecurves}. Yet, in the following section, we will give a converse statement 
in the algebraic context.

 \begin{example} \label{ex:transversality}
Consider the Lie algebra  $\g=\langle x^2\partial_x\rangle\subset\vf(\C)$. All line bundles over $\C$ are trivial, 
$\mathrm{Pic}_\g(\C)=0$, while we have $\mathrm{Pic}_\g(\C)=\C^2$. A general representative cocycle of $\Hmod^1(\g,\mathcal{O}(\mathbb C))$
has the form $\lambda(x^2\partial_x)=A+Bx$ with $A,B\in\C$, and the corresponding lifted Lie algebra is
 \[ 
\g^{\lambda} = \langle x^2\partial_x+(A+Bx)u\partial_u\rangle\subset\vf(\C\times\C).
 \] 
Generic orbits of both $\g$ and $\g^\lambda$ are 1-dimensional for any choice of $A$ and $B$, thus $\g^\lambda$
is transversal. However, the general solution of the system $x^2\partial_x(f(x)) = (A+Bx) f(x)$ is
 \[ 
f(x)=x^B e^{-A/x}.
 \]
This is a (meromorphic) $\g$-invariant divisor on $\mathbb C$ only when $A=0$ and $B \in \mathbb Z$,
i.e., $\op{Div}_\g(\C P^1)=\mathbb{Z}$ and not all equivariant line bundles come from invariant divisors.
 \end{example}

\subsection{Lie group vs Lie algebra approach} \label{sect:Gequivar}

Let $G$ be a Lie group acting on $M$. We consider the group $\mathrm{Pic}_G(M)$ of $G$-equivariant line 
bundles over $M$.  In the setting of algebraic schemes, it was studied in \cite[Ch. 1.3]{M}. 
Here we give a different description of $\mathrm{Pic}_G(M)$ emphasizing its relation to $\mathrm{Pic}_\g(M)$ when $\g$ is the Lie algebra of vector fields corresponding to the Lie group action, 
but demonstrate that, in general, $\mathrm{Pic}_G(M)$ is not isomorphic to $\mathrm{Pic}_\g(M)$.

% \begin{definition}[{\color{red} old}]
%A lift of $G$ action to a line bundle $L$ over $M$ is a projectable action of $G$ on $L$, which preserves
%the linear structure (commutes with fiber rescalings). The group of such lifts modulo $G$-equivariant bundle 
%isomorphisms is called the $G$-equivariant Picard group $\op{Pic}_G(M)$.
% \end{definition}
\begin{definition}
 A lift $\hat \rho$ of a group action $\rho \colon G \times M \to M$ to a line bundle $\pi\colon L \to M$ is a map $\hat \rho \colon G \times L \to L$ such that $\rho_g \colon L \to L$ is a vector bundle automorphism for each $g \in G$ and the following diagram commutes:
 % https://q.uiver.app/#q=WzAsNCxbMCwxLCJHXFx0aW1lcyBNICJdLFswLDAsIkcgXFx0aW1lcyBMIl0sWzEsMCwiTCJdLFsxLDEsIk0iXSxbMiwzLCJcXHBpIl0sWzEsMiwiXFxoYXQgXFxyaG8iXSxbMCwzLCJcXHJobyJdLFsxLDAsIlxcbWF0aHJte2lkfVxcdGltZXMgXFxwaSIsMl1d
\[\begin{tikzcd}
	{G \times L} & L \\
	{G\times M } & M
	\arrow["\pi", from=1-2, to=2-2]
	\arrow["{\hat \rho}", from=1-1, to=1-2]
	\arrow["\rho", from=2-1, to=2-2]
	\arrow["{\mathrm{id}\times \pi}"', from=1-1, to=2-1]
\end{tikzcd}\]
The pair $(\pi\colon L\to M,\hat \rho)$ is called a $G$-equivariant line bundle. The space of such bundles, modulo the natural equivalences, has the group structure with the operation of tensor product. The group of $G$-equivariant line bundles  $\mathrm{Pic}_G(M)$ is called the $G$-equivariant Picard group. 
 \end{definition}

We assume $G$ acts by biholomorphisms on $M$. 
The general description of $G$-equivariant line bundles over $M$ can be done in terms of a cohomology theory that 
generalizes both the \v{C}ech cohomology and the Lie group cohomology with 
coefficients in the $G$-module $\mathcal{O}^\times (M)$.

Let $\pi \colon L \to M$ be a line bundle. Assume there exists a lift of the group action to $L$, i.e.\ 
for each $\varphi \in G$ there exists a (holomorphic) vector bundle automorphism $\hat \varphi$ on $L$, 
satisfying $\pi(\hat \varphi(p)) = \varphi(\pi(p))$ for each $p \in L$  (to simplify formulas, we use the notation $\varphi = \rho_g$ and $\hat \varphi = \hat \rho_g$). 
Let $\mathcal{U}=\{U_\alpha\}$ be a trivializing chart for $L$, and $u_\alpha$ be a (linear) 
fiber coordinate on $\pi^{-1}(U_\alpha) \simeq U_\alpha \times \mathbb C$. 
Then $\hat \varphi$ acts on $u_\alpha$ in the following way:
 \begin{equation}\label{eq:Lambda}
\hat \varphi^*(u_\alpha) = \Lambda_{\alpha \beta}(\varphi) u_\beta, \qquad \Lambda_{\alpha \beta}(\varphi) 
\in  \mathcal{O}^\times\left(U_\beta \cap \varphi^{-1}(U_\alpha)\right). 
 \end{equation}
Composing with a second element in the Lie group gives
 \[
\hat \psi^*(\hat \varphi^*(u_\alpha)) = \psi^*(\Lambda_{\alpha \beta}(\varphi)) 
\Lambda_{\beta \gamma}(\psi) u_\gamma
 \]
on $U_\gamma \cap \psi^{-1}(U_\beta \cap \varphi^{-1}(U_\alpha))$.
Simultaneously, on $U_\gamma \cap \psi^{-1}(\varphi^{-1}(U_\alpha))$, we have
 \[ 
(\hat \psi^* \circ \hat \varphi^*)(u_\alpha) = \Lambda_{\alpha \gamma} (\varphi \circ \psi) u_\gamma.
 \]
Thus on $U_\gamma \cap\psi^{-1}(U_\beta) \cap\psi^{-1}(\varphi^{-1}(U_\alpha))$ we get:
 \begin{equation}\label{eq:Lambdacocycle}
\psi^*(\Lambda_{\alpha \beta}(\varphi)) \Lambda_{\beta \gamma}(\psi) =  
\Lambda_{\alpha \gamma} (\varphi \circ \psi). 
 \end{equation}
When $\varphi$ is equal to the identity transformation on $M$, equation \eqref{eq:Lambda} gives 
$\Lambda_{\alpha \beta}(\mathrm{id}) = g_{\alpha \beta}$, where $g_{\alpha \beta}$ is the transition function 
of $\pi$ on $U_\alpha \cap U_\beta$. Setting $\varphi=\mathrm{id}$ in \eqref{eq:Lambdacocycle} gives
 \[
\Lambda_{\alpha \gamma}(\psi) = \psi^*(g_{\alpha \beta})\Lambda_{\beta \gamma}(\psi)
 \]
while setting $\psi=\mathrm{id}$ leads to
 \[ 
\Lambda_{\alpha \gamma} (\varphi) = \Lambda_{\alpha \beta}(\varphi) g_{\beta \gamma}.
 \]
The last equality shows that if the transition functions are given, then $\Lambda_{\alpha \beta}(\varphi)$ on 
$U_\beta \cap U_\alpha \cap \varphi^{-1}(U_\alpha)$ is uniquely determined by $\Lambda_{\alpha \alpha}(\varphi)$. 

Next, changing the fiber coordinates $v_\alpha = \mu_\alpha u_\alpha$ for 
$\mu_\alpha \in \mathcal{O}^\times(U_\alpha)$, 
 \[
\hat \varphi^*(\mu_\alpha) \Lambda_{\alpha \beta}(\varphi) u_\beta= \hat \varphi^*(\mu_\alpha u_\alpha)=
\hat \varphi^*(v_\alpha)= \tilde \Lambda_{\alpha \beta}(\varphi) v_\beta= 
\tilde \Lambda_{\alpha \beta} (\varphi) \mu_\beta u_\beta,
 \]
results in the equivalence relation
 \[ 
\Lambda_{\alpha \beta}(\varphi) \sim \frac{\varphi^*(\mu_\alpha)}{\mu_\beta} \Lambda_{\alpha \beta}(\varphi).
 \]
Introducing the differentials 
 \begin{align*}
(D^0 \mu (\varphi))_{\alpha \beta} &= \frac{\varphi^*(\mu_\alpha)}{\mu_\beta}, \\
(D^1 \Lambda(\varphi, \psi))_{\alpha \beta \gamma} &= \frac{\Lambda_{\alpha \gamma}(\varphi \circ \psi)}{\psi^*(\Lambda_{\alpha \beta}(\varphi)) \Lambda_{\beta \gamma} (\psi)},
 \end{align*}
we see that $\Lambda=\{\Lambda_{\alpha \beta}\}$ defines a $G$-equivariant line bundle over $M$ 
if and only if $D^1 \Lambda=0$. Moreover, the $G$-equivariant line bundles defined by $\Lambda$ and 
$\tilde \Lambda$ are equivalent if and only if $\tilde \Lambda= \Lambda \cdot D^0 \mu$ for some $\mu$. 
We define the action group cohomology for a given cover $\mathcal{U}$ 
 \[
\HH^1_{\mathcal{U}}(G,\mathcal{O}^\times)= \frac{\ker(D^1)}{\mathrm{im}(D^0)}
 \]
and in general we use the direct limit of this cohomology,
$\HH^1(G,\mathcal{O}^\times):=\varinjlim \HH^1_{\mathcal{U}}(G,\mathcal{O}^\times)$.

 \begin{prop}
The group $\mathrm{Pic}_G(M)$ of $G$-equivariant line bundles is isomorphic to $H^1(G,\mathcal{O}^\times)$.
 \end{prop}

Let us note that we consider not abstract, but rather continuous (van Est) group cohomology, cf.\ \cite{F}. 
In fact, the above specifies cochains to be holomorphic.

 \begin{remark}
For a trivial line bundle we get the Lie group cohomology $\HH^1(G,\mathcal{O}^\times(M))$
of the Lie group $G$ with the values in the module $\mathcal{O}^\times(M)$.
On the other hand, with $\varphi$ and $\psi$ being $\op{id}_M$, the above definition gives the \v{C}ech cohomology 
$\Hcech^1(M,\mathcal{O}^\times)$ of $M$ with the values in the sheaf $\mathcal{O}^\times$.
Thus  $\mathrm{Pic}_G(M)$ interpolates between the two cohomologies.
 \end{remark}

Any Lie group action gives rise to a Lie algebra $\mathfrak g$ of vector fields. 
Consider a one-parameter group $\varphi_t \subset G$, and the corresponding vector field $X$. 
 %defined by \[ X_{\varphi_t} = \frac{d}{dt} \varphi_t.\]
Denote the vector field on $L$ corresponding to  $\hat \varphi_t$ by $\hat X$. For small $t$ the set 
$U_\alpha \cap \varphi_t^{-1}(U_\alpha)$ is nonempty, and on this set we have
 \[ 
\hat \varphi_t^*(u_\alpha) = \Lambda_{\alpha \alpha}(\varphi_t) u_\alpha.
 \]
When $t$ approaches $0$, then $U_\alpha \cap \varphi_t^{-1}(U_\alpha)$ approaches $U_\alpha$, 
and the Lie derivative of $u_\alpha$ with respect to $\hat X$ on $U_\alpha$ is given by
 \[ 
L_{\hat X}(u_\alpha) = \frac{d}{dt} \Big|_{t=0} \Lambda_{\alpha \alpha}(\varphi_t) u_\alpha.
 \]
Comparing this to the lifts $\hat X=X +\lambda_{\alpha}(X) u_\alpha \partial_{u_\alpha}$ 
discussed in Section \ref{sect:2.1} results in the relation
 \[ 
\lambda_\alpha(X) =  \frac{d}{dt} \Big|_{t=0} \Lambda_{\alpha \alpha}(\varphi_t).
 \]
Thus a $G$-equivariant line bundle on $M$ yields a $\g$-equivariant line bundle for $\g=\op{Lie}(G)$. 
However, the map
 \[ 
\mathrm{Pic}_G(M) \to \mathrm{Pic}_\g(M)
 \]
in general is neither injective nor surjective. 
Non-injectivity is illustrated by an action of a discrete group, like $\mathbb Z_m:z\mapsto z^m$ on $\C P^1$.
Non-surjectivity is demonstrated as follows.

 \begin{example}[Projective action revisited]\label{ExgG}
The Lie groups $SL(2,\C)$ and $PGL(2,\C)$ act on $\C P^1$ by M\"obius transformations.
In the open cover given by charts $U_0\simeq\C(x)$ and $U_\infty\simeq\C(y)$ the action is
 \[ 
\varphi=\left(\begin{matrix} a & b \\ c & d \end{matrix}\right):\quad
\varphi^*(x) = \frac{ax+b}{cx+d}, \qquad \varphi^*(y) = \frac{dy+c}{by+a}.
 \]
For $SL(2,\mathbb C)$ the lifts are given by 
 \[ 
\hat\varphi^*(u_0)=\frac{u_0}{(cx+d)^A}, \qquad \hat\varphi^*(u_\infty) = \frac{u_\infty}{(by+a)^A},
 \]
where $A\in \mathbb Z$, as in Example \ref{ex:sl2P1}. On the other hand, for  $PGL(2,\mathbb C)$ the lifts are given by
 \[ 
\hat\varphi^*(u_0)=\frac{(ad-bc)^{A/2} u_0}{(cx+d)^A}, \qquad \hat\varphi^*(u_\infty) = \frac{(ad-bc)^{A/2} u_\infty}{(by+a)^A},
 \]
which is well-defined if and only if $A=2m\in 2\mathbb Z$.
%, because then $(cx+d)^{2m} = (-cx-d)^{2m}$. 
In other words, the line bundle 
$\mathcal{O}_{\mathbb CP^1}(1)$ is not $PGL(2,\mathbb C)$-equivariant, 
but $\mathcal{O}_{\mathbb CP^1}(2)$ is.

This example, borrowed from \cite[Ch.\ 1.3]{M}, works in any dimension $n$: the line bundle
$\mathcal{O}_{\C P^n}(k)$ is $PGL(n+1,\C)$-equivariant iff $k\in(n+1)\mathbb Z$, i.e., 
this group lifts only to the powers of the canonical bundle $K_{\C P^n}$. 
On the other hand, all bundles $\mathcal{O}_{\C P^n}(k)$ are $SL(n+1,\C)$-equivariant. 
(Note that the center of $SL(n+1,\C)$ is $\mathbb Z_{n+1}$ and $PGL(n+1,\C)=SL(n+1,\C)/\mathbb Z_{n+1}$.)
This difference can not be seen at the Lie algebra level, 
since the two Lie group actions give rise to the same Lie algebra of vector fields.
Summarizing we have:
 \[
\op{Pic}_{SL(n+1,\C)}(\C P^n)=  \op{Pic}_{\mathfrak{sl}(n+1,\C)}(\C P^n)=\mathbb Z= \op{Pic}(\mathbb CP^n)
\supset (n+1)\mathbb Z=\op{Pic}_{PGL(n+1,\C)}(\C P^n).
 \]
 \end{example}

 Note that in this example both groups $PGL(n+1,\C)$ and $SL(n+1,\C)$ are algebraic, so this example illustrates a general result in \cite[Cor.\,1.6]{M} on $G$-linearization of high powers $L^m$ of an algebraic line bundle $L$. Next we discuss a similar effect for invariant divisors. 
 
Recall that an algebraic Lie algebra is $\g=\op{Lie}(G)$ for an algebraic Lie group $G$.  If $M$ is an algebraic variety, we  call a Lie algebra $\g \subset \vf(M)$ algebraic if it is the Lie algebra of an algebraic action by an algebraic Lie group on $M$.   The following is a converse to Proposition \ref{prop:orbitdim}
in the algebraic context (there is a version of this statement for $\op{Pic}_G(M)$).

 \begin{theorem}\label{th:orbitdimR}
Let $(L,\hat \g) \in \mathrm{Pic}_\g(M)$ be a $\g$-equivariant line bundle over an algebraic variety $M$ for an algebraic Lie algebra $\g$ of vector fields. 
Assume that the lift $\hat \g$ is algebraic and transversal.
Then there exists an integer $m\in\mathbb Z_+$ such that $L^m\in \op{im}(\Phi_1\circ  j_\g)$, i.e.,
$L^m=[D]$ for some invariant divisor $D$ with weight $\lambda$, and $\hat \g=\g^{\lambda/m}$. 
 \end{theorem}

 \begin{proof}
Since $\g$ is transversal, it admits on $L$ an absolute invariant $I=I(x,u)$, with $x$ coordinate on $M$
and $u$ a fiber coordinate on $L$, such that $\p_u(I)\not\equiv0$. This complements absolute invariants $J=J(x)$
obtained by pullback from $M$. By Rosenlicht's theorem \cite{R2} the algebraicity of the action implies that 
the invariant $I$ can be chosen rational in proper (local) variables $x,u$ (on $U_\alpha$ with algebraic overlaps).
Decompose $I$ into its Laurent series by the fiber variable $u$
 \begin{equation}\label{LS}
I=\sum_{k=-N}^\infty h_k(x)u^k.
 \end{equation}
Since $[u\p_u,\hat\g]=0$ we get that $(u\p_u)^r(I)$ is an absolute invariant for every $r$.
The spectrum of the operator $u\p_u$ on generators $u^k$ is simple, and due to rationality the coefficients of $I$
are determined by a finite number of base functions $h(x)$. Thus every term in the series \eqref{LS}
is an absolute invariant. Choose such invariant of the lowest (in absolute value) degree by $u$.
This degree $m$ does not depend on local coordinate chart $U_\alpha,\alpha\in A$, we are using, and we get:
 \[
I_\alpha=\frac{u_\alpha^m}{f_\alpha(x)}\quad\ \Longrightarrow\ \quad
1=\frac{I_\alpha}{I_\beta}=\frac{u_\alpha^m/f_\alpha(x)}{u_\beta^m/f_\beta(x)}
=g^m_{\alpha\beta}\frac{f_\beta}{f_\alpha}\quad\text{ on }\quad U_\alpha\cap U_\beta.
 \]
The collection of functions $\{f_\alpha\in\mathcal{O}(U_\alpha):\alpha\in A\}$ defines a $\g$-invariant divisor $D$
with weight $\lambda_\alpha(X)=X(\log f_\alpha)$, $X\in\g$, and the corresponding line bundle $[D]$
has transition functions
 \[
\tilde{g}_{\alpha\beta}= \frac{f_\alpha}{f_\beta}=g_{\alpha\beta}^m.
 \]
Thus $[D]=L^m$ and the claim follows.
 \end{proof}
 
 % This makes ``if ond only if'' version of Theorem \ref{XXX} in the algebraic context.

 \begin{example}
Consider the Lie algebra $\g=\langle x\p_x\rangle$ on $\C$ and its lift 
$\hat\g=\langle x\p_x+C u\p_u\rangle$ on the trivial line bundle $\C\times\C$.
It is algebraic if $C=\frac{p}{q}\in\mathbb Q$ with absolute invariant $I=\frac{u^m}{x^{Cm}}$
being algebraic for minimal $m=q$, i.e., $I=\frac{u^q}{x^p}$.
Such a situation occurs for differential invariants of curves in Euclidean plane with respect to the motion group, 
namely for the ``square of the curvature'', see the end of Introduction in \cite{KL}.  The $\g$-equivariant line bundle $(\mathbb C \times \mathbb C, \hat \g)$ is in $\mathrm{im}(j_\g)$ if and only if $C \in \mathbb Z$.
 \end{example}

\section{Invariant polynomial divisors on algebraic bundles}\label{sect:bundles}

In this section we will consider Lie algebras of vector fields on bundles that have additional structure on the fibers, 
and where it makes sense to consider divisors that are polynomial in the fiber coordinates. More precisely, 
we will focus on affine bundles in Section \ref{sect:aff} and on jet bundles in Section \ref{sect:jets}. 
In the remaining three subsections, we will apply the obtained results to examples involving Lie algebras 
of vector fields on jet spaces.

\subsection{Lie algebra action on affine bundles}\label{sect:aff}

Let $\pi \colon E \to M$ be an affine bundle (of rank $r\ge1$), and let $\hat \g \subset \vf_{\mathrm{proj}}(E)$ 
be a Lie algebra of projectable vector fields on $E$ that preserves the affine structure in the fibers. 
In this section we will focus on $\hat \g$-invariant divisors whose restriction to fibers are polynomials. We define for $U \subset M$
 \[ 
\mathfrak{P}(U) = \{f \in \mathcal{O}(\pi^{-1}(U)) \mid  f|_{\pi^{-1}(p)} 
\text{ is a polynomial for every } p \in U\},
 \]
where polynomiality is checked in affine coordinates on $E$. 
This space of functions is preserved by automorphisms: 
If $\varphi \colon E \to E$ is an automorphism of affine bundles and 
$f \in \mathfrak{P}(U)$, then $\varphi^*f \in \mathfrak{P}(\varphi^{-1}(U))$.

Assume that $\mathcal{U}=\{U_\alpha\}$ is an open cover for $M$ such that 
$\pi^{-1}(U_\alpha)\simeq U_\alpha\times\C^r$ for any affine bundle $\pi$. Then $\{\pi^{-1}(U_\alpha)\}$ 
is open cover for the total space $E$.
 
 \begin{definition}
A divisor $D=\{f_\alpha\}$ on an affine bundle $\pi\colon E \to M $ is called polynomial if its defining functions 
$f_\alpha \in \mathcal{O}( \pi^{-1}(U_\alpha))$ can be chosen to be in $\mathfrak{P}(U_\alpha)$.
 \end{definition}

In local coordinates $x^i, u^j$ on $\pi^{-1}(U_\alpha) \simeq U_\alpha \times \mathbb C^r$, 
$f_\alpha$ takes the form
 \[ 
f_\alpha = \sum_{|\sigma|\leq s} F_{\sigma}(x) u^{\sigma}, \qquad F_\sigma \in \mathcal{O}(U_\alpha).
 \]
Here $u^\sigma=\prod(u^i)^{m_i}$ for the multi-index $\sigma=(m_1\dots m_r)$.
The defining functions of a polynomial divisor $D$ satisfy $g_{\alpha \beta} = f_\alpha/f_\beta \in \mathcal{O}^\times(\pi^{-1}(U_\alpha) \cap \pi^{-1}(U_\beta))$, where both the numerator and denominator are polynomials 
in $u^1, \dots, u^r$. It follows that the polynomials must cancel each other out, which implies that the transition 
functions $g_{\alpha \beta} = f_\alpha/f_\beta$ are the pullback of functions 
$\tilde g_{\alpha \beta}\in\mathcal{O}^\times(U_\alpha \cap U_\beta)$. Thus we obtain the following proposition.

 \begin{prop} \label{prop:13}
Let $D=\{f_\alpha\}$ be a polynomial divisor on the affine bundle $\pi \colon E \to M$. 
Then $[D] = \pi^*L$ for some line bundle $L \to M$.
 \end{prop}

In the above argument, it is clear that the degree $s$ can be taken to be the same for each $\alpha$. We call the smallest such $s$ the degree of the polynomial divisor $D$.

Next we let $\hat \g \subset \vf_{\mathrm{proj}}(E)$ be a projectable Lie algebra of vector fields on $E$ 
preserving the affine structure on fibers, and consider $\hat \g$-invariant polynomial divisors.

 \begin{prop} \label{prop:aff}
Let $\pi \colon E \to M$ be a an affine bundle and let $\hat \g \subset \vf_{\mathrm{proj}}(E)$ be a projectable 
Lie algebra of vector fields on $E$ preserving the affine structure on fibers; $\g=d\pi(\hat \g)$. 
If $D$ is a $\hat \g$-invariant polynomial divisor on $E$, then $[D]=\pi^* L$ for some 
$\g$-equivariant line bundle $L \to M$.
 \end{prop}

 \begin{proof}
What remains to be proven is that the bundle $\pi \colon L \to M$ with transition functions $\tilde g_{\alpha \beta}$ 
admits a $\g$-lift. In local coordinates $x^i, u^j$ on $\pi^{-1}(U_\alpha) \simeq U_\alpha \times \mathbb C^r$, 
each $X \in \hat \g$ takes the form $X = a^i(x) \partial_{x^i} + (b_0^j(x)+b_l^j(x) u^l) \partial_{u^j}$. 
Consider an invariant divisor $D$ given by $f_\alpha = \sum_{|\sigma| \leq s} F_{\sigma}(x) u^\sigma$, 
a polynomial in $u^1, \dots, u^r$ of degree $s$. We have
 \[ 
X(f_\alpha) = \lambda_\alpha(X) f_\alpha. 
 \]
Looking at the coordinate form of $X$, it is clear that $X(f_\alpha)$ is a polynomial in $u^1, \dots, u^r$ and, 
furthermore, that its degree is $\leq s$. Thus $\lambda_\alpha(X)=X(f_\alpha)/f_\alpha$ is a rational function in 
$u^1, \dots, u^r$, and it is defined everywhere on $U_\alpha \times \mathbb C^r$ only if $\lambda_\alpha(X)$ 
is the pullback of a function $ \tilde \lambda_\alpha(X) \in \mathcal{O}(U_\alpha)$. Thus the $\hat \g$-equivariant 
line bundle over $E$ defined by $(\{g_{\alpha \beta}\}, \{\lambda_\alpha\})$ is the pullback of the 
$\g$-equivariant line bundle over $M$ defined by $(\{\tilde g_{\alpha \beta}\}, \{\tilde \lambda_\alpha \})$.
 \end{proof}

Proposition \ref{prop:aff} (which was reformulated in Theorem \ref{Th3}) is relevant, for instance, for investigation of relative invariants of tensor fields  (and other geometric objects like affine connections) on a manifold $M$ under the action of some Lie algebra $\g \subset \vf(M)$.

 \begin{example}
Consider the bundle $E=S^2 T^* M \to M$ whose sections are symmetric 2-forms on $M$, and let $\g$ be the 
Lie algebra of holomorphic vector fields on $M$. There is a canonical lift $\hat \g \subset \vf(E)$ of $\g$. 
Let $x^1, \dots, x^n$ be coordinates on $M$ and $u_{11}, u_{12}, \dots, u_{nn}$ be the additional induced 
coordinates on $E$.  The function $\det([u_{ij}])$, with $u_{ji}=u_{ij}$ for $i<j$, is the local expression for 
a $\g$-invariant divisor, and its (local) weight is $-2 \mathrm{div}_{dx^1 \wedge \cdots \wedge dx^n}$. 
Globally, the line bundle given by this $\g$-invariant divisor is the pullback of the line bundle 
 $(\Lambda^n T^*M)^{\otimes 2} = (K_M)^{\otimes 2}$ over $M$.
 \end{example}

Let us make a brief remark about invariant rational divisors. Each such is a ratio of two invariant polynomial divisors. The weights of invariant rational divisors form a lattice generated by weights of invariant polynomial divisors. In other words, we have the following relation: 
\begin{equation} \label{eq:divrational}
 \mathrm{Span}_{\mathbb Z} \left(\mathrm{Div}_\g^{\mathrm{pol}}(M) \right) =\mathrm{Div}_\g^{\mathrm{rat}}(M).
 \end{equation}

 \subsection{Lie algebra action on jet bundles}\label{sect:jets}

Now we consider polynomial divisors on jet bundles. Most of the arguments here closely resemble those 
in Section \ref{sect:aff}, but some additional care must be taken. Our introduction to jets will be very brief, 
and we refer to \cite{KL2, KV, O} for a more comprehensive treatment.  

Let $J^k(E,m)$ denote the space of $k$-jets of codimension-$m$ submanifolds of $E$, and $J^k \pi$ the space of 
$k$-jets of sections of the fiber bundle $\pi$. In statements that are true for both $J^k(E,m)$ and $J^k \pi$, 
we will use the notation $\J^k$ which can always be replaced with either of the two (an exception to this convention occurs only in Section \ref{sect:ExC}). There are natural bundle 
structures $\pi_{k,l}\colon \J^k \to \J^l$ for $0 \leq l \leq k$, and $\pi_k\colon J^k \pi \to M$.

Coordinates on  $J^k \pi$ and $J^k(E,m)$  are induced from coordinates on the total space of $\pi$ or $E$, 
respectively. Given a bundle $\pi \colon E \to M$, and an open cover  $\{U_\alpha\}$ of coordinate charts of $E$, the collection $\{\pi_{k,0}^{-1}(U_\alpha)\}$ is 
an open cover of $J^k \pi$. The split coordinates $x^1,\dots x^n, u^1, \dots, u^m$ on $U_\alpha$ 
induce additional canonical coordinates $u_\sigma^j$, $|\sigma|\leq k$ where $\sigma$ is a multi-index, on $\pi_{k,0}^{-1}(U_\alpha)$. 

To get an open cover of $J^k(E,m)$, we let $\{U_\alpha\}$ be an open cover  of coordinate charts of $E$ that trivializes the bundle 
$J^1(E,m) \to E$. On each $U_\alpha$,  for a given set of $m+n$ coordinates, we choose a splitting $x^1,\dots, x^n, u^1, \dots, u^m$, 
and we denote the corresponding coordinate chart on $J^1(E,m)$ by $U_\alpha^{i_1 \cdots i_m}$ with 
$1 \leq i_1 < \cdots < i_m \leq \dim E$. For each way of splitting there is one chart.  
The split coordinates on $U_\alpha$ induce additional canonical coordinates $u_\sigma^i$ ($|\sigma|\leq k$)
on $\pi_{k,1}^{-1}(U_\alpha^{i_1 \cdots i_m})$ for $k \geq 1$. 
The collection $\{\pi_{k,1}^{-1}(U_\alpha^{i_1 \cdots i_m})\}$ is an open cover of $J^k(E,m)$.
We define for $U \subset \J^i$
 \[
\mathfrak{P}_i(U) = \{f \in \mathcal{O}(\pi_{\infty,i}^{-1}(U)) \mid  f|_{\pi_{\infty,i}^{-1}(p)} 
\text{ is a polynomial for every } p \in U\}.
 \]
Polynomiality is defined with respect to the canonical coordinates described above. 
For example, in the case of $J^k(E,m)$, then $f \in \mathfrak{P}_j(U)$ if and only if we have, for each $\alpha$,
 \[
f\Bigr|_{\pi_{\infty,j}^{-1}\left(U \right) \cap \pi_{\infty,1}^{-1}\left(U_\alpha^{i_1 \cdots im}\right)}
 = \sum_{j \leq |\sigma|\leq k, |\tau| \leq r} F^\sigma_\tau u_\sigma^\tau
 \] 
for some $k,r \in \mathbb Z_{\geq 0}$ and a collection $\{F^\sigma_\tau\}$ of functions on $J^j(E,m)$.

If $\varphi \colon E \to E$ is a point transformation and $f \in \mathfrak{P}_i(U)$, then 
$(\varphi^{(\infty)})^* f \in \mathfrak{P}_i(\varphi^{-1}(U))$ for $i \geq 1$. Similarly, if $\pi\colon E \to M$ 
is a fiber bundle,  $\varphi \colon E \to E$ is a fiber-preserving biholomorphism, and $f \in \mathfrak{P}_i(U)$, 
then $(\varphi^{(\infty)})^* f \in \mathfrak{P}_i(\varphi^{-1}(U))$ for $i \geq 0$. In particular, 
point transformations preserve $\mathfrak{P}_1$ while fiber-preserving transformations do the same for 
$\mathfrak{P}_0$. Based on this, we introduce the following notion.

 \begin{definition}
A divisor on $J^k(E,m)$ is called polynomial if its defining functions $f_\alpha^{i_1 \cdots i_m} \in \mathcal{O}(\pi^{-1}_{k,1}
(U_\alpha^{i_1 \cdots i_m}))$ can be chosen to be in $\mathfrak{P}_1(U_\alpha^{i_1 \cdots i_m})$. If $\pi \colon E \to M$ is a fiber bundle, a divisor on 
$J^k \pi$ is called polynomial if its defining functions $f_\alpha \in \mathcal{O}(\pi^{-1}_{k,0}(U_\alpha))$ 
can be chosen to be in $\mathfrak{P}_0(U_\alpha)$.
 \end{definition}

 \begin{prop} \label{prop:15}
{\rm (1)} Let $\pi \colon E \to M$ be a fiber bundle, and let $D=\{f_\alpha\}$ be a polynomial divisor on $J^k \pi$. 
Then $[D] = \pi_{k,0}^*L$ for some line bundle $L \to E=J^0 \pi$.\\
{\rm (2)} Let $E$ be a manifold and $D = \{ f_{\alpha}^{i_1 \cdots i_m} \}$ be a polynomial divisor on $J^k(E,m)$. 
Then $[D] = \pi_{k,1}^* L$ for some line bundle $L \to J^1(E,m)$.
 \end{prop}

 \begin{proof}
The proofs of (1) and (2) are very similar, so we prove only (2). 
Since $f_\alpha^{i_1 \cdots i_m}/f_\beta^{j_1 \cdots j_m}$ are elements in 
$\mathcal{O}^\times(\pi_{k,1}^{-1} (U_\alpha^{i_1 \cdots i_m} \cap U_\beta^{j_1 \cdots j_m}))$, 
the polynomial parts are required to cancel. Thus  the transition functions 
$f_\alpha^{i_1 \cdots i_m}/f_\beta^{j_1 \cdots j_m}$ are the pullback of elements in 
$\mathcal{O}^\times (U_\alpha^{i_1 \cdots i_m} \cap U_\beta^{j_1 \cdots j_m})$, 
which are the transition functions of a line bundle over $J^1(E,m)$.
 \end{proof}

From the proof it follows that the order and degree of the polynomials $f_\alpha^{i_1 \cdots i_m}$ and $f_\beta^{j_1 \cdots j_m}$ agree. Therefore, the order and degree are also well-defined notions for $D = \{f_\alpha^{i_1 \cdots i_m}\}$, and this is true also for divisors on $J^k \pi$. We define the weighted degree of the monomial $c(x,y,y_i) y_{\sigma_1}^{j_1} \cdots y_{\sigma_s}^{j_s} \in \mathfrak{P}_1(U_\alpha^{i_1 \cdots i_m})$ (with $|\sigma_l| \geq 2$ for each $l$) to be $\sum_{l=1}^s |\sigma_l|$, and the weighted degree of a sum of such to be the maximal weighted degree of its monomial parts. The weighted degree can be defined for a divisor in the same way that order and degree were defined above.  (In the case when $E$ is a bundle, the weighted degree also counts the first-order jet variables $y_i^j$.)

Next, consider a Lie algebra of vector fields $\g\subset\vf(\J^0)$;
in the case $\J^0=J^0\pi$ assume also that $\g$ is $\pi$-projectable. 
The Lie algebra prolongs to a unique Lie algebra $\g^{(k)} \subset \vf(\J^k)$, see, for instance, \cite[Sec.\ 1.5]{KL2}. We are interested in polynomial $\g^{(k)}$-invariant divisors on $\J^k$.

 \begin{prop} \label{prop:polynomialonjets}
{\rm (1)} Let $\g$ be a Lie algebra of projectable vector fields on a fiber bundle $\pi\! : E\to M$ and 
let $D$ be a  $\g^{(k)}$-invariant polynomial divisor on $J^k \pi$. Then $[D]=\pi_{k,0}^*L$ for 
some $\g$-equivariant line bundle $L\to E$.\\
{\rm (2)} Let $\g$ be a Lie algebra of point vector fields on $J^0(E,m)$ and $D$ a  $\g^{(k)}$-invariant polynomial
divisor on $J^k(E,m)$. Then $[D]=\pi_{k,1}^*L$ for some $\g^{(1)}$-equivariant line bundle $L\to J^1(E,m)$.
 \end{prop}

 \begin{proof}
On $J^k(E,m)$ we will use the split coordinate charts $\pi_{k,1}^{-1}(U_\alpha^{i_1 \cdots i_m})$. If $\pi\colon E\to M$ 
is a bundle, then the splitting is canonical and we use the charts $\pi^{-1}_{k,0}(U_\alpha)$ on $J^k \pi$.

The prolongation of a vector field $X=a^i \partial_{x^i} + b^j \partial_{y^j} \in \mathfrak g \subset \vf(E)$ is  
given by
 \begin{equation*}
X^{(k)} =a^i \partial_{x^i} + \sum_{0 \leq |\sigma|\leq k} b_\sigma^j \partial_{y_\sigma^j}
 \end{equation*}
where $b_{\sigma}^j$ are given recursively by (see \cite[Th. 3.4]{KV})
 \begin{equation}
b_{\sigma i}^j = D_{x^i}(b_\sigma^j)-y_{\sigma l}^j D_{x^i}(a^l).\label{eq:ProlongationFormula}
 \end{equation}
When $|\sigma|=d$ it is clear that $b^j_{\sigma}$ is a sum of monomials of the form 
$c(x,y) y^{j_1}_{\sigma_1} \cdots y^{j_s}_{\sigma_s}$ with $|\sigma_l| \leq d$ for each $l$ and 
$|\sigma_1|+\cdots + |\sigma_s| \leq d+1$. Thus the weighted degree of $b_\sigma^j$ is $\leq d+1$.

If $D=\{f_\alpha^{i_1 \cdots i_m}\}$ is a polynomial invariant divisor on $J^k(E,m)$ of weighted degree $d$, 
then for any $X \in \g$ the function $X^{(k)}(f_\alpha^{i_1 \cdots i_m})$ has weighted degree $\leq d+1$. 
The equality $X^{(k)}(f_\alpha^{i_1 \cdots i_m}) = \lambda_\alpha^{i_1 \cdots i_m}(X) f_\alpha^{i_1 \cdots i_m}$ implies that  $\lambda_\alpha^{i_1 \cdots i_m}(X) = X^{(k)}(f_\alpha^{i_1 \cdots i_m})/f_\alpha^{i_1 \cdots i_m}$ is holomorphic on $\pi_{k,1}^{-1}(U_\alpha^{i_1 \cdots i_m})$ if and only if the polynomial parts of the denominator is canceled out by the numerator. In this case, $\lambda_\alpha^{i_1 \cdots i_m}(X)$ is polynomial of weighted degree $\leq 1$, meaning that it is the pullback of a function $\tilde{\lambda}_\alpha^{i_1 \cdots i_m}(X) \in \mathcal{O}(U_\alpha^{i_1 \cdots i_m})$. 
 Thus we see that the $\g^{(k)}$-equivariant line bundle over $J^k(E,m)$ 
defined by the pair $(\{f_\alpha^{i_1 \cdots i_m}/f_\beta^{j_1 \cdots j_m}\}, \{\lambda_\alpha^{i_1 \cdots i_m}\})$ 
is the pullback of a $\g^{(1)}$-equivariant line bundle over $J^1(E,m)$.

If $E \to M$ is a fiber bundle, then we get a similar argument, but now $\lambda_\alpha(X)$ has weighted degree $0$, and  is therefore the pullback 
of a function in $\mathcal{O}(U_\alpha)$.
 \end{proof}

This proposition, whose second part was reformulated in Theorem \ref{Th4}, tells us that invariant polynomial divisors on $\J^k$ are sections of pullbacks of equivariant line bundles over 
$\J^1$ or $\J^0$. In particular, they are controlled by $\HH^1(\mathrm{Tot}^\bullet(C))$, where 
$C^{p,q}=C^{p,q}(\g^{(1)},\{U_\alpha^{i_1\cdots i_m}\})$ or $C^{p,q}=C^{p,q}(\g,\{U_\alpha\})$ or,  more precisely, by $\op{Pic}_{\g^{(r)}}(\J^r)$ for $r=1,0$ respectively.
It is remarkable that this fact is independent of the order $k$ (one should compare to the statement 
of the Lie-B\"acklund theorem \cite{KV,O}, although the proofs are different).

If the bundle $\pi$ has, in addition, an affine structure, then we can consider divisors with local defining 
functions $f_\alpha$ in
 \[
\mathfrak{P}_{-1}(U_\alpha) = \{f \in \mathcal{O}(\pi_{\infty}^{-1}(U_\alpha)) \mid  
f|_{\pi_{\infty}^{-1}(p)} \text{ is a polynomial for every } p \in U_\alpha\},
 \]
i.e., divisors that are polynomial on fibers of $\pi_k \colon J^k \pi \to M$. These are preserved under 
($k^{\text{th}}$-prolongation of) morphisms of affine bundles, and we will refer to them as ``polynomial divisors'' 
in this context. For such a divisor $D$, we have $[D]=\pi_k^* L$ for some line bundle $L \to M$. We can apply 
the same ideas as above to obtain the following result, which we leave without proof.

 \begin{prop} \label{prop:affjet}
Let $\g$ be a Lie algebra of projectable vector fields on an affine bundle $\pi \colon E \to M$ that preserves 
the affine structure, and $D$ be a $\g^{(k)}$-invariant polynomial divisor on $J^k \pi$. Then $[D]=\pi_k^*L$ 
for some $d\pi(\g)$-equivariant line bundle $L \to M$.
 \end{prop}

 \begin{example}[Riemannian geometry]
Let $\pi\colon S^2_+T^*M\to M$ denote the bundle of nondegenerate symmetric 2-forms on $M$ and  $\g$ the Lie algebra of holomorphic vector fields on $M$, which induces a Lie algebra $\g^{(k)}$ of vector fields on $J^k \pi$ for $k=0,1,\dots$. If $D$ a polynomial $\g^{(k)}$-invariant divisor on $J^k \pi$, then $[D]$ is the pullback of a line bundle $L\to M$. 
For example, if $D$ is the divisor on $J^2 \pi$ that is given locally by the numerator of the scalar curvature of the metric, then 
 $[D]=\pi_{2}^*(\Lambda^nT^*M)^{\otimes4}$.
 \end{example}
 
 Computations of invariants in jets often result in rational relative differential invariants, which are related to polynomial differential invariants via a jet analogue of formula \eqref{eq:divrational}. This will be demonstrated in the following examples.

\subsection{Example A: Three-dimensional Heisenberg algebra on the plane}

Consider the following Lie algebra of vector fields on the plane:
 \[ 
\mathfrak g=\langle \partial_x,\partial_y, y\partial_x\rangle \subset \vf(\mathbb C^2).
 \]
It has the structure relations of the Heisenberg algebra and it prolongs naturally to the Lie algebra 
$\g^{(1)}$ of vector fields on  $J^1(\mathbb C^2,1)$. Choosing $y$ as the dependent variable gives
 \[ 
\mathfrak g^{(1)}|_{U_1} = \langle \partial_x, \partial_y, y \partial_x-y_1^2 \partial_{y_1} \rangle,
 \]
where $U_1 \subset J^1(\mathbb C^2,1)$ denotes the open chart determined by our choice of dependent variable on $\mathbb C^2$. Taking instead $x$ as the dependent variable results in a different chart $U_2 \subset J^1(\mathbb C^2,1)$ where the prolongation of $\mathfrak g$ takes the form
 \[ 
\mathfrak g^{(1)}|_{U_2} = \langle \partial_x, \partial_y, y\partial_x+\partial_{x_1} \rangle.
 \]
These two charts cover $J^1(\mathbb C^2,1)=U_1 \cup U_2$. 
On overlap $U_1 \cap U_2$ we get $(x,y,y_1)\equiv(x,y,1/x_1)$. 
 % The coordinates are on $U_1 \cap U_2$ related by $(x,y,y_1) = (x,y,1/x_1)$. 
In each of the two charts we compute the Chevalley-Eilenberg cohomology:
 \[ 
\HH^1( \mathfrak g^{(1)}, \mathcal{O}(U_1)) = \mathbb C^2, \qquad 
\HH^1( \mathfrak g^{(1)}, \mathcal{O}(U_2)) = 0.
 \]
A representative $\lambda_1$ of a general element  in $\HH^1(\g^{(1)},\mathcal{O}(U_1))$ takes the form
 \[ 
\lambda_1(\partial_x)=0, \qquad \lambda_1(\partial_y)=0, \qquad 
\lambda_1 (y \partial_x-y_1^2 \partial_{y_1}) = A + B y_1, \qquad A,B \in \mathbb C.
 \]
The compatibility condition $\lambda_1(X) - \lambda_2(X) = X(g_{12})/g_{12}$, $\forall X \in \g^{(1)}$ gives 
the general transition function $g_{12} = C y_1^{-B} e^{A/y_1}$. This function is holomorphic on $U_1 \cap U_2$ 
if and only if $B \in \mathbb Z$. Changing the representative $(g,\lambda)\in C^{0,1}\times C^{1,0}$ by 
the coboundary $\partial^0\mu$ where $\mu_1=1, \mu_2=C e^{A x_1}$, we get $g_{12} = y_1^{-B}$ and
 \[ 
\lambda_1 = (0,0,A+By_1), \qquad \lambda_2 = (0,0,A).
 \]
Thus $\op{Pic}_{\g^{(1)}}(\J^1)=\C\times\mathbb Z\to\op{Pic}(\J^1)=\mathbb Z$ is epimorphic.

We identify a generating set $(I,\nabla)$ of absolute differential invariants in charts as follows:
 \[
\Bigl(-\frac{y_2}{y_1^3},\frac{1}{y_1}D_x\Bigr)\ \text{ on }\ U_1\quad\longleftrightarrow\quad 
\bigl(x_2,D_y\bigr)\ \text{ on }\ U_2.
 \]

The invariant divisors on $J^1(\mathbb C^2,1)$ are generated by $f_1=y_1, f_2=1$ of weight $(A,B)=(0,-1)$. 
Note that the invariant ODE $y_1=0$ is not visible from the local computations on $U_2$.  
Indeed, its solutions are $y=\mathrm{const}$ for the independent variable $y$, which are not graphs $x=h(y)$.

General $\g^{(2)}$-invariant divisors on $J^2(\mathbb C^2,1)$ are generated by $f=\{f_1,f_2\}$ and 
the absolute invariant $I$. In particular, the irreducible invariant submanifolds of codimension 1 in $\J^2$ 
are given by the divisors $\tilde f=\{\tilde{f}_1,\tilde{f}_2\}=\{y_2-C y_1^3,x_2+C\}$ of weight  $(A,B)=(0,-3)$,  parametrized by $C \in \mathbb C$.

Note that the non-zero parameter $A$ above is not realizable by an invariant divisor (on $\J^1$ such are $y_1^{-B}$).
Higher prolongations give no new weights of polynomial divisors  and we conclude, with the help of Proposition \ref{prop:polynomialonjets},
 \[
\mathbb Z=  j_{\g^{(\infty)}}\op{Div}_{\g^{(\infty)}}^{\op{rat}}(\J^\infty)\subset
\op{Pic}_{\g^{(1)}}(\J^1)=\C\times\mathbb Z.
 \]

\subsection{Example B: Invariant divisors of curves in the  projective plane} \label{sect:projectivecurves}

Consider the Lie algebra $\mathfrak{sl}(3,\C)\subset\vf(\C P^2)$ of projective vector fields. 
Differential invariants of curves in  the projective plane were studied already in 1878 by Halphen in his PhD thesis 
\cite{H} (see also the recent treatment \cite{Projective} in the real case). In this section we demonstrate how 
the framework developed in this paper sheds new light on those classical invariants.

The manifold $\mathbb CP^2$ is covered by the three charts $U_i=\mathbb CP^2 \setminus \{z_i =0\}$, $i=1,2,3$, 
where $[z_1:z_2:z_3]$ are homogeneous coordinates. Let us start by focusing on $U_3$ with coordinates 
$x=z_1/z_3, y=z_2/z_3$. In these local coordinates we have
 \begin{align*}
\mathfrak{sl}(3,\mathbb C)|_{U_3} = \langle \partial_x, \partial_y, y \partial_x, x\partial_y, x \partial_x-y \partial_y, x \partial_x+y\partial_y, x^2 \partial_x+xy\partial_y, xy\partial_x+y^2\partial_y\rangle.
  \end{align*}

\subsubsection{Equivariant line bundles} 

The cohomology group $\HH^1(\mathfrak{sl}(3,\C),\mathcal{O}(U_3))=\C$ was computed in 
\cite[Table 3]{GKO}, and  also in \cite{S}. Our global approach shows that
 \[
\mathrm{Pic}_{\mathfrak{sl}(3,\mathbb C)}(\mathbb CP^2) \simeq \mathrm{Pic}(\mathbb CP^2) = 
\left\{\mathcal{O}_{\mathbb CP^2}(k) \mid k \in \mathbb Z \right\}\simeq\mathbb Z.
 \]
Skipping the details of this computation, we instead focus on the corresponding computation in $J^1(\C P^2,1)$.
Choosing $y$ as the ``dependent'' variable we get an open coordinate chart  $U_{3}^y \subset J^1(U_3,1)$ 
in which the prolonged vector fields take the form
 \begin{equation}\begin{gathered}\label{sl3pr}
\hskip-7pt X_1=\partial_x, \ \ X_2=\partial_y, \ \ X_3=y \partial_x-y_1^2 \partial_{y_1}, \ \
 X_4 = x \partial_y+\partial_{y_1}, \ \ X_5 = x \partial_x-y\partial_y-2 y_1 \partial_{y_1}, \hskip-8pt \\
\hskip-7pt X_6 = x \partial_x+y\partial_y, \ \ X_7 = x^2 \partial_x+xy \partial_y+(y-xy_1)\partial_{y_1}, \ \
 X_8 = xy\partial_x+y^2 \partial_y+(y-xy_1)y_1 \partial_{y_1}. \hskip-7pt
 \end{gathered}\end{equation}
Let us start by computing $\HH^1(\mathfrak{sl}(3,\C)^{(1)},\mathcal{O} (U_3^y))$. 
For a general cocycle $\lambda_3^y$, we define 
 \[
a_i(x,y,y_1) := \lambda_3^y(X_i)\in \mathcal{O}(U_3^y).
 \] 
By subtracting a coboundary we can set $a_1= 0$. The cocycle condition involving $X_1$ and $X_2$ implies that 
$\partial_x(a_2)=0$, and by subtracting a coboundary (now $x$-independent) we set $a_2 = 0$. 
The eight cocycle conditions
 \[ 
X_i(a_j)-X_j(a_i)-\lambda([X_i,X_j])=0, \qquad 1\leq i\leq2,\ 3\leq j\leq6
 \]
reduce to $X_i(a_j)=0$ and imply that $a_3, a_4, a_5, a_6$ are independent of $x$ and $y$.

By subtracting a coboundary (independent of $x$ and $y$) we set $a_4(y_1) = 0$.
Then for the PDE system defined by the remaining cocycle conditions, we get the general holomorphic solution:
 \begin{gather*}
a_1 = 0, \quad a_2 = 0, \quad a_3= A_2 y_1, \quad a_4= 0, \quad a_5=A_2, \quad a_6=A_1, \\ 
a_7 = \frac{3A_1+A_2}{2} x, \quad a_8 = A_2 x y_1+\frac{3A_1-A_2}{2} y,
 \end{gather*}
from which we see that $\HH^1(\mathfrak{sl}_3^{(1)}, \mathcal{O}(U_3^y))=\mathbb C^2$.

A similar computation can be done in the open coordinate chart $U_3^x \subset J^1(U_3,1)$, where 
$x$ is the dependent variable. In these coordinates, related to the previous by $x_1=1/y_1$ on overlap 
$U_3^y \cap U_3^x$, the prolonged vector fields take the form
 \begin{gather*}
X_1=\partial_x, \quad   X_2=\partial_y, \quad   X_3=y \partial_x+ \partial_{x_1}, \quad   X_4 = x \partial_y-x_2^2\partial_{x_1}, \quad   X_5 = x \partial_x-y\partial_y+2 x_1 \partial_{x_1}, \\
X_6 = x \partial_x+y\partial_y, \quad \tilde X_7 = x^2 \partial_x+xy \partial_y+(x-yx_1)x_1 \partial_{x_1}, \quad   X_8 = xy\partial_x+y^2 \partial_y+(x-yx_1)\partial_{x_1}.
 \end{gather*}
Defining $b_i(y,x,x_1):=  \lambda_3^x(X_i) \in \mathcal{O}(U_3^x)$, and repeating the computations above, a general representative of an element in $\HH^1(\mathfrak{sl}_3^{(1)},\mathcal{O}(U_3^x))$ is given by
 \begin{gather*}
 b_1 = 0, \quad b_2 = 0, \quad b_3=0, \quad b_4=\tilde A_2 x_1, \quad b_5=-\tilde A_2, \quad b_6=\tilde A_1, \\ b_7 = \tilde A_2 y x_1+\frac{3\tilde A_1-\tilde A_2}{2} x, \quad b_8 = \frac{3 \tilde A_1+\tilde A_2}{2} y,
 \end{gather*}
implying $\HH^1(\mathfrak{sl}(3,\mathbb C)^{(1)}, \mathcal{O}(U_3^x))=\mathbb C^2$.

The compatibility condition $\lambda_3^y(X)-\lambda_3^x(X)= X(g_{33}^{yx})/g_{33}^{yx}$ implies that 
$\tilde A_1=A_1$ and $\tilde A_2=A_2$. In this case the transition function has the form 
$g_{33}^{yx} = Cy_1^{-A_2}$, and it is holomorphic if and only if $A_2 \in \mathbb Z$.  
The constant $C$ can be set equal to $1$ via a suitable $\mathfrak{sl}(3,\mathbb C)^{(1)}$-invariant 
$\delta^{0,0}$-coboundary. Thus we conclude 
$\op{Pic}_{\mathfrak{sl}(3,\C)^{(1)}}(J^1(U_3,1))=\C\times\mathbb Z$.

Next, we perform similar computations on the remaining charts $U_1^x, U_1^y, U_2^x, U_2^y$ of $J^1(\C P^2,1)$. 
In $U_2\subset\C P^2$ we have coordinates $(\tilde x, \tilde y)=(z_1/z_2, z_3/z_2)$. Choosing $\tilde y$ 
as dependent variable results in coordinates $(\tilde x,\tilde y,\tilde y_1)$ on $U_2^y\subset J^1(\C P^2,1)$.  
On $U_3^y \cap U_{2}^y$ we have $x=\tilde x/\tilde y$, $y=1/\tilde y$ and 
$y_1=\tilde y_1/(\tilde x\tilde y_1-\tilde y)$. In these coordinates, the generators of 
$\mathfrak{sl}(3,\mathbb C)^{(1)}$ are:
 \begin{gather*}
X_1 = \tilde y \partial_{\tilde x} - \tilde y_1^2 \partial_{\tilde y_1}, \quad 
X_2 = - \tilde x \tilde y \partial_{\tilde x} - \tilde y^2 \partial_{\tilde y} + (\tilde x \tilde y_1-\tilde y) \tilde y_1 \partial_{\tilde y_1}, \quad 
X_3 = \partial_{\tilde x}, \\  
X_4 = - \tilde x^2 \partial_{\tilde x} - \tilde x \tilde y \partial_{\tilde y}+ (\tilde x \tilde y_1 - \tilde y) \partial_{\tilde y_1}, \quad 
X_5= 2 \tilde x \partial_{\tilde x}+\tilde y \partial_{\tilde y}-\tilde y_1 \partial_{\tilde y_1}, \\ 
X_6 = - \tilde y \partial_{\tilde y} - \tilde y_1 \partial_{\tilde y_1}, \quad 
X_7 = -\tilde x \partial_{\tilde y} - \partial_{\tilde y_1}, \quad X_8 = - \partial_{\tilde y}.
 \end{gather*}
Defining $c_i(\tilde x, \tilde y, \tilde y_1) := \lambda_{2}^y(X_i) \in \mathcal{O}(U_{2}^y)$ 
yields a general element in $\HH^1(\mathfrak{sl}(3,\mathbb C)^{(1)}, \mathcal{O}(U_{2}^y))$:
 \begin{gather*}
c_1 = \frac{3 B_1 + B_2}{2} \tilde y_1,\quad 
c_2 =  - \frac{3 B_1+B_2}{2} \tilde x  \tilde y_1 + \frac{3 B_1-B_2}{2} \tilde y, \quad c_3= 0, \\  
c_4= -B_2 \tilde x, \quad c_5= B_2, \quad c_6= B_1, \quad c_7= 0, \quad c_8= 0.
 \end{gather*}
The compatibility condition $\lambda_{3}^y(X) - \lambda_{2}^y(X) = X(g_{32}^{yy})/g_{32}^{yy}$ implies that 
$B_1=(A_2-A_1)/2$ and $B_2=(3A_1+A_2)/2$. The transition function on $U_3^y \cap U_{2}^y$ is given by 
$g_{32}^{yy} = \tilde C \tilde y^{-B_2} (\tilde x \tilde y_1 - \tilde y)^{A_2}$. It is holomorphic if and only if 
$A_2,B_2 \in \mathbb Z$. To sum up, we have 
 \begin{equation}\label{A2A}
A_2\in\mathbb Z\ \text{ and }\ (3A_1+A_2)/2\in\mathbb Z.
 \end{equation}
By doing a similar analysis on the intersection of the remaining charts, one gets 
 \begin{equation}\label{sl3CCP2}
\mathrm{Pic}_{\mathfrak{sl}(3,\mathbb C)^{(1)}} (J^1(\mathbb CP^2,1)) = \mathbb Z^2.
 \end{equation}
Furthermore, the map $\op{Pic}_{\mathfrak{sl}(3,\C)^{(1)}} (J^1(\C P^2,1))\to\op{Pic}(J^1(\C P^2,1))$ 
is injective, since we have  $\Hmod^1(\mathfrak{sl}(3,\C)^{(1)},J^1(\C P^2,1))=0$.

Let us compare this to known bundles over $J^1(\mathbb CP^2,1)$, starting with canonical bundles. The line 
bundle $\Lambda^3 T^* J^1(\mathbb CP^2,1) \to J^1( \mathbb CP^2,1)$ corresponds to $(A_1,A_2)=(-2,2)$, 
while the pullback of the line bundle $\Lambda^2T^*\C P^2\to\C P^2$ via $\pi_{1,0}\colon J^1(\C P^2,1)\to\C P^2$ 
corresponds to $(A_1,A_2)=(-2,0)$. This is easy to check by computing divergences of $X_1,\dots,X_8$ with respect to 
the volume forms  $\Omega_0= dx \wedge dy$ and $\Omega_1 = dx \wedge dy \wedge dy_1$ on 
$U_3 \subset \mathbb CP^2$ and $U_3^y \subset J^1(\mathbb CP^2,1)$, respectively: $\mathrm{div}_{\Omega_0}$ corresponds to $(A_1,A_2)=(2,0)$ and $\mathrm{div}_{\Omega_1}$ corresponds to $(A_1,A_2)=(2,-2)$.
(Note that divergences with respect to different volume forms differ (locally) by a coboundary in the modified Chevalley-Eilenberg complex.)

Furthermore, the pullback of the line bundle $\mathcal{O}_{\mathbb CP^2}(1) \to \mathbb CP^2$ corresponds 
to $(A_1,A_2)=(2/3,0)$ because of the relation between the canonical and tautological bundles over $\C P^2$ 
(see Remark \ref{rk:tautological}). The vertical bundle $VJ^1(\C P^2,1) \subset TJ^1(\C P^2,1)$ 
corresponds to $(A_1,A_2)=(0,-2)$, while the subbundle  $\langle\omega\rangle\subset T^*J^1(\C P^2,1)$ 
defined by the contact form $\omega\in\Gamma(T^*J^1(\C P^2,1))$ corresponds to $(A_1,A_2)=(-1,1)$. 
The subset $(A_1,A_2)\subset\C^2$ satisfying \eqref{A2A} is generated by the elements $(2/3,0)$ and $(-1,1)$. 
This leads to the following concrete description:
 
 \begin{prop} \label{prop:picsl3}
Consider the standard realization of $\mathfrak{sl}(3,\C)\subset\vf(\C P^2)$, and its prolongation 
$\mathfrak{sl}(3,\C)^{(1)}\subset\vf(J^1(\C P^2,1))$. The equivariant Picard group \eqref{sl3CCP2} is
 \[
\mathrm{Pic}_{\mathfrak{sl}(3,\mathbb C)^{(1)}} (J^1(\mathbb CP^2,1))= 
\left\{\langle\omega\rangle^{\otimes k_1}\otimes \pi_{1,0}^*\mathcal{O}_{\C P^2}(k_0)  
\mid k_0, k_1\in\mathbb Z\right\}\simeq\mathbb Z^2.
 \]
The integer parameters are related to the above weights like this: $A_1=-k_1+\frac23k_0, A_2=k_1$.
 \end{prop}

 % Notice that Proposition \ref{prop:picsl3} together with Proposition \ref{prop:polynomialonjets} lets us reformulate 
 % the problem of finding invariant divisors in terms of solving a discrete set of PDEs. Any polynomial invariant divisor 
 % on $J^k(\mathbb CP^2,1)$ restricted to  $\pi_{k,1}^{-1}(U_3^y)$ has weight $\lambda_3^y$ for some 
 % $[(g,\lambda)] \in \mathrm{Pic}_{\mathfrak{sl}(3,\mathbb C)^{(1)}} (J^1(\mathbb CP^2,1))$. 
 % Thus it is (locally) the solution to the system
 % \[ X^{(k)} (f_3^y) = \lambda_3^y(X) f_3^y, \qquad X \in \mathfrak{sl}(2,\mathbb C),\]
 % where $\lambda_3^y$ can now be taken as known input. Not all of these systems will have solutions, and 
 % Proposition \ref{prop:orbitdim} can be used to exclude some of these cases, at least when $k$ is sufficiently small.

\subsubsection{Invariant divisors and absolute differential invariants}

Generators for the absolute differential invariants are well-known, see e.g.\ \cite[Table 5]{O}. 
The field of rational absolute differential invariants is generated by
 \[
\Bigl( I_7=\frac{R_7^3}{R_5^8}, \nabla=\frac{R_2 R_7}{R_5^3} D_x\Bigr)
 \]
where $R_2, R_5, R_7$ are expressed in the following way on $\pi_{7,1}^{-1}(U_3^y)$:
\begin{align*}
  R_2 &= y_2,  \\
  R_5 &= 9y_2^2y_5-45y_2 y_3 y_4+40 y_3^3, \\
  R_7 &= 18 y_2^4 (9 y_2^2 y_5 - 45 y_2 y_3 y_4 + 40 y_3^3) y_7 - 189 y_2^6 y_6^2
   + 126 y_2^4 (9 y_2 y_3 y_5 + 15 y_2 y_4^2 - 25 y_3^2 y_4) y_6\\ & -  189 y_2^4 (15 y_2 y_4
   + 4 y_3^2) y_5^2 + 210 y_2^2 y_3 (63 y_2^2 y_4^2    - 60 y_2 y_3^2 y_4+32 y_3^4) y_5 -  4725 y_2^4 y_4^4 \\
  & - 7875 y_2^3 y_3^2 y_4^3
   + 31500 y_2^2 y_3^4 y_4^2 - 33600 y_2 y_3^6 y_4 + 11200 y_3^8 .
\end{align*}
We use a different set of generators than \cite{O} in order to obtain rational invariants, which by \cite{KL} are 
sufficient to separate orbits in general position. Table 5 in \cite{O} also contains the Lie determinant $R_2 R_5^2$ 
on the locus of which the orbit dimension drops. The Lie algebra $\mathfrak{sl}(3,\mathbb C)^{(6)}$ acts 
simply transitively on the complement of $\{R_2 R_5=0\} \subset \pi_{6,1}^{-1}(U_3^y)$; note that 
$\dim J^6(\C P^2,1)=\dim\mathfrak{sl}(3,\mathbb C)$. 
A complete description of the orbit structure (over $\mathbb R$) can be found in \cite{Projective}.

 \begin{remark}
Proposition \ref{prop:orbitdim} gains the following insight. Computing orbit dimensions
of $\g=\mathfrak{sl}(3,\C)$ in $J^4(\C P^2,1)$ shows that an invariant divisor exist only if 
$A_2 = 3 A_1$, in which case it is $y_2^{-A_1}$, but this is meromorphic only if $A_1\in\mathbb Z$. 
For $k\geq 5$, the generic orbit dimension of $\g^{(k)}$ on $J^k(\C P^2,1)$ is the same as that of 
$(\g^{(k)})^{\lambda}$, independently of $\lambda$. The general invariant divisor is given by 
$R_2^{2A_1-A_2}  R_5^{(A_2-3 A_1)/6}$, however this function is meromorphic if and only if 
$(2A_1-A_2), (A_2-3A_1)/6 \in\mathbb Z$. Together with \eqref{A2A} this implies that weights 
$(A_1,A_2)$ belong to the lattice generated by $(3,-3)$ and $(2,0)$.
 \end{remark}

The polynomials $R_2, R_5, R_7$ are local expressions, defined on $\pi_{7,1}^{-1}(U_3^y)$, for invariant 
polynomial divisors. But they extend uniquely to polynomial divisors on $J^7(\mathbb CP^2,1)$.  
For $R_2$, $R_5$ and $R_7$, the weight $\lambda_3^y$ is given by $(A_1,A_2)=(-1,-3)$, $(A_1,A_2)=(-6,-12)$ 
and $(A_1,A_2)=(-16,-32)$, respectively.  In particular, $R_2$ and $R_5$ do not combine to a rational 
absolute differential invariant (weight $0$), which is consistent with the fact that $\g^{(6)}$ has an open orbit on 
$J^6(\mathbb CP^2,1)$. It is also clear that $R_2$ and $R_5$ are local generators for polynomial invariant 
divisors on $J^6(\mathbb CP^2,1)$ since they generate a 2-dimensional space of weights.

Combining weights of the invariant divisors, we obtain the above absolute invariant $I_7$
together with the following invariant meromorphic tensor fields:
 \[ 
\alpha_5 = \frac{R_5}{R_2^4} dx \wedge dy \in \Gamma(\pi_{5,0}^*\Lambda^2 T^* \mathbb CP^2), \qquad \alpha_7=\frac{R_7}{R_2^3 R_5^2} (dy-y_1 dx) \in \Gamma(\pi_{7,1}^{*}\langle \omega \rangle).
 \]
The inverse bivector $\alpha_5^{-1}=\frac{R_2^4}{R_5}D_x \wedge\partial_y$ contracted with $\alpha_7$ 
gives the invariant derivation $\nabla$ above.
 % completing the set of generators $I_7$, $\nabla$ of the field of rational absolute differential invariants

 \begin{remark}
These tensor fields can be compared to those of Theorem 5.1 of \cite{Projective}. Their $R_2^{-3} R_5^{2/3} 
(dy-y_1 dx)$ is multi-valued over $\C$, but its cube is the rational invariant tensor $I_7^{-1}\alpha_7^3$.
 \end{remark}

 Note that  in general polynomial divisors $\op{Div}^{\text{pol}}_\g(M)$ determine a weight sub-monoid in $\op{Pic}_\g(M)$,
while rational divisors $\op{Div}^{\text{rat}}_\g(M)$, obtained as ratios of the former, determine a lattice.

 \begin{theorem}
The lattice generated by polynomial divisors for $\g=\mathfrak{sl}(3,\C)$ acting on $J^\infty(\C P^2,1)$
is a sublattice of order 3 in the equivariant Picard group on 1-jets:
 \[
\mathbb Z^2 \simeq j_{\g^{(\infty)}}\Bigl(\op{Div}^{\text{\rm rat}}_{\g^{(\infty)}}\bigl(J^\infty(\C P^2,1)\bigr)\Bigr)
\subsetneq \op{Pic}_{\g^{(1)}}\bigl(J^1(\C P^2,1)\bigr)\simeq \mathbb Z^2.
 \]
 \end{theorem}

This is basically a summary of the computations. Indeed, from the tensor fields $\alpha_5$,  $\alpha_7$ 
we see that (pullbacks of) line bundles in $\op{Pic}_{\mathfrak{sl}(3,\C)^{(1)}}(J^1(\C P^2,1))$ are realized 
as $[D]$ for some rational $\mathfrak{sl}(3,\mathbb C)^{(7)}$-invariant divisor $D$ on $J^7(\mathbb CP^2,1)$ 
when $k_0/3,k_1 \in \mathbb Z$, where $k_0$ and $k_1$ are the parameters used in Proposition \ref{prop:picsl3}. 
To understand why $\mathcal{O}_{\mathbb CP^2}(1)$ is not realized in this way 
one must consider which Lie group is acting here. The Lie algebra $\g=\mathfrak{sl}(3,\C)$ on $\C P^2$
integrates to the Lie group $G=PGL(3,\C)$, and then results from Example \ref{ExgG} apply.

 \begin{remark}
Non-degenerate curves in $\C P^n$ up to projective transformations $\g=\mathfrak{sl}(n+1,\C)$ 
were studied by Wilczynski \cite{Wi}. He computed fundamental differential invariants 
via the correspondence with linear ordinary differential equations of order $n+1$.
Our results generalize to give two-dimensional lattice $\op{Pic}_{\g^{(1)}}\bigl(J^1(\C P^n,1)\bigr)$,
 which constrains the weights of relative differential invariants.
%Our results generalize to give two-dimensional lattice $\op{Pic}_{\g^{(1)}}\bigl(J^1(\C P^n,1)\bigr)$
%for values of relative differential invariants.
 \end{remark}

\subsection{Example C: Second-order ODEs modulo point transformations revisited}\label{sect:ExC}

Finally for $h\in\mathcal{O}(U)$, $U\subset J^1(\mathbb C P^2,1)$, we consider scalar second-order ODEs 
\begin{equation}
\{y_2=h(x,y,y_1)\}\subset J^2(\C P^2,1) \label{eq:ODE}
\end{equation}
together with the Lie algebra sheaf $\g=\vf(\J^0)$ of germs of holomorphic vector fields on $\J^0=\C P^2$. 
Here and throughout this section we use the notation $\J^s=J^s(\mathbb C P^2,1)$, while $J^k(\J^1)$ consists of $k$-jets 
of functions $h$ on $\J^1$. Our goal is to find generators for the  invariant divisors on $J^4(\J^1)$. 

Relative invariants were first found by A.\ Tresse in \cite{T} via Lie theory 
 % using extensive and sophisticated computations
and then by E.\ Cartan via his theory of moving frames \cite{Ca}.
We apply our global framework to justify the (two-dimensional) weight lattice introduced in \cite{K1}
and generate relative invariants for this classical problem in a novel and conceptually transparent manner.

Any vector field on $\J^0$ prolongs uniquely to a vector field on $\J^2$. This action induces 
an (infinitesimal) transformation on the space of second-order ODEs. Choose local coordinates $x,y$ on $\C P^2$,
denote $p=y_1,u=y_2$ the induced coordinates on $\J^2$, then an ODE is a hypersurface $u=h(x,y,p)$
in $J^0(\J^1)=\J^2$. Redefining $\g$ to be the image (prolongation) of $\vf(\J^0)$ in $\J^2$,
its further prolongation, the Lie algebra $\mathfrak g^{(k)} \subset \vf(J^k(\J^1))$, 
is spanned by the vector fields of the form
 \begin{equation} \label{eq:odevf}
a D_x+b D_y+c D_p+ \sum_{|\sigma|\leq k} D_\sigma^{(k)}(\psi) \p_{u_\sigma}
 \end{equation}
where $a,b$ are functions of $x,y$, $c=(\partial_x+p\partial_y)\varphi$ for $\varphi = b-pa$,
$D_x$ is the operator of total derivative by $x$ and similar for $D_y,D_p$, while $D_\sigma$ is their
composition for multi-indices of variables (see \cite{K1}),
and the function $\psi$ is given by
 \[ 
\psi= (\partial_x+p\partial_y)^2\varphi+ u(\partial_y\varphi-2(
\partial_xa+p\partial_ya)-a u_x-b u_y-c u_p.
 \]

The Lie algebra $\mathfrak g^{(0)}=\g$ preserves the fibers of the affine bundle $J^0(\J^1)\to\J^1$   (and their affine structure). 
Thus, in order to compute invariant divisors that are polynomial on fibers of  $J^k(\J^1)\to\J^1$,
we exploit Proposition \ref{prop:affjet} and start with classification of $\g$-equivariant line bundles on $\J^1$.

\subsubsection{$\g$-equivariant line bundles}

In Example B we saw that the $\mathfrak{sl}(3,\C)^{(1)}$-equivariant line bundles on 
$\J^1$ were generated by the line bundles $\pi_{1,0}^*\mathcal{O}_{\mathbb CP^2}(1)$ 
and $\langle\omega\rangle\subset T^*\J^1$. 
Since $\mathfrak{sl}(3,\C)^{(1)}\subset\g$, we have a natural homomorphism
 \begin{equation}\label{eq:restrsl3}
\op{Pic}_\g(J^1(\C P^2,1))\to\op{Pic}_{\mathfrak{sl}(3,\C)^{(1)}}(J^1(\C P^2,1)).
 \end{equation}

 \begin{prop}\label{PpPp}
Homomorphism \eqref{eq:restrsl3} is an isomorphism.
 \end{prop}

 \begin{proof}
We first prove that \eqref{eq:restrsl3} is surjective. 
Clearly, the bundle $\mathcal{O}_{\mathbb CP^2}(-3) \simeq \pi_{1,0}^* \Lambda^2 T^* \mathbb CP^2$ 
admits a $\g$-lift, due to naturality of the cotangent bundle. The bundle $\langle\omega\rangle\subset T^*\J^1$ 
admits a $\g$-lift since the prolongation preserves the Cartan distribution $\op{Ann}(\omega)\subset T \J^1$. 
What remains to be seen is that $\mathcal{O}_{\C P^2}(1)$ admits a $\g$-lift. 
On $\mathcal{O}_{\C P^2}(1)$, the local weight $\lambda_3$ of a general vector field 
$X=a(x,y) \partial_x+b(x,y)\partial_y$ on $U_3$ (for example) is $\lambda_3(X)=(a_x + b_y)/3$, 
and it is not difficult to check that this extends to a compatible weight for each $X \in \g$.

Now we prove injectivity. Let $[(g,\lambda)]\in\op{Pic}_\g(\J^1)$  be in the kernel of \eqref{eq:restrsl3}. 
%If \eqref{eq:restrsl3} maps $[(g,\lambda)]$ to $0\in\op{Pic}_{\mathfrak{sl}(3,\C)^{(1)}}(\J^1)$, then 
Then $[g]=0\in\op{Pic}(\J^1)$ and there exists a representative for $[\lambda]$ such that  
$\lambda|_{\mathfrak{sl}(3,\C)^{(1)}} =0$.

Take an arbitrary point in $\J^1$ and choose a chart with coordinates centered at this point (origin). 
Due to transitivity of $\mathfrak{sl}(3,\C)^{(1)}$ on $\J^1$ we can assume, without loss of generality, 
that the coordinate chart is $U_3^y$ from Section \ref{sect:projectivecurves}. 
We will compute $\lambda|_{U_3^y}$. It is clear that if $\lambda(X) \neq 0$ for some $X\in\g$, then  
$\lambda(X)|_{U_3^y}\neq 0$ since $U_3^y\subset\J^1$ is a dense subset.

We continue with the notation from Section \ref{sect:projectivecurves}, so that 
$\mathfrak{sl}(3,\C)^{(1)}|_{U_3^y} = \langle X_1, \cdots, X_8 \rangle$ with $X_i$ given by 
\eqref{sl3pr}. We have $\lambda(X_1)=\cdots =\lambda(X_8)=0$. Next, consider the vector fields
 \begin{gather*}
Y_1 = x^2 \partial_y+2x\partial_{y_1} ,  \quad Y_2 = x^2\partial_x-2xy \partial_y - (4xy_1+2y)\partial_{y_1}, \\
Y_3 = y^2 \partial_y -2xy \partial_x+(2xy_1^2+4yy_1)\partial_{y_1},  \quad Y_4=y^2 \partial_x-2yy_1^2 \partial_{y_1}.
 \end{gather*}
The commutation relations
 \begin{gather*}
  [X_1,Y_1]= 2 X_4, \quad [X_2,Y_1]=0, \quad [X_4,Y_1]=0, \\
  [X_1,Y_2]= 2 X_5, \quad [X_2,Y_2]=-2X_4, \quad [X_4,Y_2]=-3 Y_1, \\
  [X_1,Y_3]= -2 X_3, \quad [X_2,Y_3]=-2X_5, \quad [X_4,Y_3]=-2 Y_2, \\
  [X_1,Y_4]=0, \quad [X_2,Y_4]=2 X_3, \quad [X_4, Y_4]=-Y_3, \\
  [X_6, Y_i]=Y_i, \qquad i=1,2,3,4,
 \end{gather*}
give four differential equations on each function $\lambda(Y_i)$, implying $\lambda(Y_1)=\cdots=\lambda(Y_4)=0$.

Furthermore, all polynomial vector fields are generated by $X_1,\dots, X_8$ and $Y_1, \dots, Y_4$. 
Indeed, for $j \geq 3$ we have
\begin{align*}
  x^i y^{j-i} \partial_x &= \frac{1}{i-3} [x^2 \partial_x,x^{i-1} y^{j-i} \partial_x], \qquad i \neq 0,3, \\
  y^{j} \partial_x &= \frac{1}{j-1} [y^2 \partial_y, y^{j-1} \partial_x],\\
  x^3 y^{j-3} \partial_x &= \frac{1}{j-2} \left([x^2 \partial_y,xy^{j-2} \partial_x] +2x^2 y^{j-2} \partial_y\right),
\end{align*}
and by swapping $x$ and $y$ we also generate $x^i y^{j-i} \partial_y$ for $i=0,\dots,j$.
Thus all vector fields with polynomial coefficients of degree $\geq 3$ are of the form $[Z,Y]$, 
where the coefficients of $Y$ have degree $2$ and the coefficients of $Z$ have degree strictly lower than 
those of $[Z,Y]$. Then the general cocycle condition
 \[ 
\lambda([X,Y]) =X(\lambda(Y))-Y(\lambda(X)) 
 \]
implies that $\lambda(X)=0$ for any polynomial vector field $X$ on $U_3^y$. 

On any compact subset $K\subset U_3^y$, the subspace of vector fields in $\vf(K)$ with polynomial coefficients 
is dense in $\vf(K)$. It follows that $\lambda(X)|_K = 0$ for every $X \in \g$ for any $K$, and hence that 
$\lambda(X)|_{U_3^y}=0$ for every $X \in \g$. Thus $\lambda=0$.
 \end{proof}

\subsubsection{Invariant divisors}

Now we compute the $\g^{(4)}$-invariant divisors on $J^4(\J^1)$. 
Let us work in the coordinate chart $\tau_{4}^{-1}(U_3^y)$, where $\tau_4$ denotes the projection 
$\tau_{4}\colon J^4(\J^1)\to\J^1$. From Proposition \ref{PpPp}, we know that 
$[\lambda]\in H^1(\g,\mathcal{O}(U_3^y))$ has a representative of the form
 \[ 
\lambda=C_0 \mathrm{div}_{dx \wedge dy}+ C_1 \mathrm{div}_{dx \wedge dy \wedge dy_1},
 \]
where $(C_0,C_1)$ is related to $(A_1,A_2)$ by $A_1=2(C_0+C_1)$ and $A_2=-2C_1$. Condition \eqref{A2A}
is equivalent to $3 C_0 ,2 C_1 \in \mathbb Z$. If $f$ is a general polynomial of some fixed degree, then the system
 \[ 
X^{(k)} f = \lambda(X) f, \qquad X \in \g^{(0)}
 \]
reduces to a linear system on the coefficients of $f$ for each choice of $(C_0,C_1)$. By sequentially setting 
$C_0=0, \pm 1/3, \pm 2/3, \dots$ and $C_1=0, \pm 1/2, \pm 1, \dots$ and letting $f$ be a general polynomial 
of degree 3 with undetermined coefficients, we get a series of linear systems determining the coefficients of the 
polynomial. In this way, we obtain the solutions
\begin{align*}
  f_1 &= u_{pppp}, \\
  f_2 &=u_{xxpp} + 2 p u_{xypp} + 2 u u_{xppp} + p^2 u_{yypp} + 2 p u u_{yppp} + u^2 u_{pppp} + (u_{y} u_{ppp} - u_{p} u_{ypp} - 4 u_{yyp}) p  \\ &- 3 u u_{ypp} + (-u_{xpp} + 4 u_{yp}) u_{p} + u_{x} u_{ppp} - 3 u_{y} u_{pp} + 6 u_{yy} - 4 u_{xyp},
\end{align*}
which  have weights $(C_0, C_1)= (2,-5/2)$ and $(-2,1/2)$, respectively. 
Computing the rank of prolonged vector fields at generic point we conclude that
the action of $\g^{(4)}$ has an open orbit in $J^4(\J^1)$. Thus there are no (nonconstant) absolute invariants on $\J^4$. 
Now, if $f_3$ was another invariant divisor of  general weight $(C_0,C_1)=(2A-2B,(B-5A)/2)$ with rational $A,B$, 
then for some integer $m$ the ratio
 \[ 
\frac{f_3^m}{f_1^{Am} f_2^{Bm}}
 \]
 is a rational function with weight $(0,0)$ and hence is an absolute differential invariant,  and therefore constant. 
Hence $f_3^m$ is proportional to $f_1^{Am} f_2^{Bm}$. 

Taking into account Proposition \ref{prop:affjet} 
we conclude the following.

 \begin{theorem}
The lattice generated by polynomial divisors for the Lie algebra $\g=\vf(\J^0)$ acting on $J^\infty(\J^1)$ is a sublattice in the equivariant Picard group on 1-jets:
 \[
\mathbb Z^2\simeq j_{\g^{(\infty)}}\Bigl(\op{Div}^{\text{\rm { rat }}}_{\g^{(\infty)}}\bigl(J^\infty(\J^1)\bigr)\Bigr)
\subsetneq \op{Pic}_{\g}\bigl(\J^1\bigr)\simeq \mathbb Z^2.
 \]
 \end{theorem}
 
Let us note that cohomology of line bundles was explored in \cite{HK} to compute Cartan invariants of projective connections, which correspond to a particular class of ODEs of the form \eqref{eq:ODE} with $h$ cubic in $y_1$; our methods though are quite distinct.

\section{Outlook}

In this work we proposed a theory of global scalar relative differential invariants, based on familiar
notions of divisors and line bundles. While $G$-equivariant line bundles were known for algebraic and
compact groups, the more general notions of equivariant Picard group $\op{Pic}_\g(M)$ and invariant 
divisor group $\op{Div}_\g(M)$ for a Lie algebra $\g$ appear to be new and have certain subtleties. 
(These notions even extend to Lie algebra sheaves, as seen in Example C.)

The basic setup is analytic, but we also consider polynomial divisors in affine bundles. 
Such bundles arise in successive jet-prolongation, and polynomial relative differential invariants 
are natural and sufficient in the equivalence problem of invariant hypersurfaces. We thus explore polynomial divisors in jet spaces.
%We have $j\bigl(\op{Div}(\J^\infty)\bigr)=j\bigl(\op{Div}(\J^1)\bigr)$ in $\op{Pic}(\J^\infty)=\op{Pic}(\J^1)$. In the case of a fiber (resp.\ affine) bundle $\pi$ this can be pushed down to $\J^0$ (resp.\ $M$). 
While $j\bigl(\op{Div}(\J^\infty)\bigr)=j\bigl(\op{Div}(\J^1)\bigr)$ in $\op{Pic}(\J^\infty)=\op{Pic}(\J^1)$ 
(in the case of fiber/affine bundle $\pi$ this can be pushed down to $\J^0$, resp.\ $M$),
the $\g$-equivariant counterpart is more complicated. In general, 
$\op{Pic}_{\g^{(\infty)}}(\J^\infty)\neq\op{Pic}_{\g^{(1)}}(\J^1)$, and similarly for
invariant divisors. However weights of invariant polynomial divisors are 1-jet determined, as Propositions
\ref{prop:polynomialonjets} and \ref{prop:affjet} state. This gives an effective bound on
multipliers for relative invariants and, in many cases, an algorithmic approach to compute them.

%Singular $\g$-orbits of higher codimensions are related
Invariant submanifolds of higher codimensions are related, in the same manner, to higher rank equivariant  vector bundles.
While there are no  general tools that classify analytic/algebraic  vector bundles of higher rank, some part of the theory generalizes.
Weights of vector-valued relative invariants are matrix-valued cocycles, leading to a more general cohomology theory.
%These more general invariants form a module over the algebra of relative invariants.

Lastly, there is  a differential algebra aspect to the theory of invariant divisors on jet bundles. The structure theory of these global relative differential invariants will be discussed elsewhere.

\bigskip
\bigskip

\textbf{Acknowledgments.}
The research leading to our results has received funding from the Norwegian Financial Mechanism 2014-2021
(GRIEG project SCREAM, registration number 2019/34/H/ST1/00636) and the Tromsø Research Foundation
(project “Pure Mathematics in Norway”), as well as UiT Aurora project MASCOT. The research of E.S. was partially 
funded by COST Action CaLISTA CA21109 supported by COST (European Cooperation in Science and Technology).

%%%%%%%%%%

\end{document}